\documentclass[a4paper,10pt]{amsart}
\usepackage[english]{babel}
\usepackage{amsmath}
\usepackage{layout}
\usepackage{amssymb}
\usepackage{cite}
\usepackage{amsthm}
\usepackage{a4wide}
\usepackage{graphicx}
\usepackage[all]{xy}
\usepackage{mathrsfs}

\newcommand{\p}{\mathbb{P}}

\newcommand{\HH}{\mathbb{H}}

\newcommand{\ol}{\overline}

\newcommand{\mcX}{\mathcal{X}}

\newcommand{\rH}{\mathrm{H}}

\newcommand{\Ar}{\mathrm{Ar}}

\DeclareMathOperator{\Fal}{Fal}
\DeclareMathOperator{\gr}{gr}
\DeclareMathOperator{\inff}{inf}
\DeclareMathOperator{\fin}{fin}

\DeclareMathOperator{\SL}{SL}

\DeclareMathOperator{\Spec}{Spec}

\DeclareMathOperator{\divv}{div}
\DeclareMathOperator{\Cl}{Cl}
\DeclareMathOperator{\Div}{Div}

\newcommand{\h}{\mathbb{H}}

\newcommand{\F}{\mathbb{F}}
\newcommand{\Z}{\mathbb{Z}}

\newcommand{\R}{\mathbb{R}}
\newcommand{\Q}{\mathbb{Q}}

\newcommand{\C}{\mathbb{C}}

\newcommand{\mcY}{\mathcal{Y}}

\newcommand{\Qbar}{\overline{\mathbb{Q}}}

\newcommand{\bs}{\backslash}

\newcommand{\g}{\mathfrak{g}}

\newcommand{\isomlto}{\;\tilde{\longrightarrow}\;}

\newcommand{\Pic}{\textrm{Pic}}



\newtheorem{thm}{Theorem}[subsection]

\newtheorem{lem}[thm]{Lemma}

\newtheorem{prop}[thm]{Proposition}
\newtheorem{cor}[thm]{Corollary}

\theoremstyle{definition}

\newtheorem{defn}[thm]{Definition}

\newtheorem{opm}[thm]{Remark}

\makeatletter
\def\cleardoublepage{\clearpage\if@twoside \ifodd\c@page\else
\hbox{} \thispagestyle{plain}
   \newpage
\if@twocolumn\hbox{}\newpage\fi\fi\fi} \makeatother
\clearpage{\pagestyle{plain}\cleardoublepage}

\begin{document}



\baselineskip=17pt

\newif\ifamslatex \amslatexfalse

\ifamslatex

\title[Polynomial bounds for Arakelov invariants]{Polynomial bounds for Arakelov invariants of curves with given Belyi degree}
\author[Ariyan Javanpeykar]{Ariyan Javanpeykar \\\\ with an appendix by
Peter Bruin}

\else

\title{Polynomial bounds for Arakelov invariants of Belyi curves}

\author{Ariyan Javanpeykar \\\\ with an appendix by Peter Bruin}

\address{Mathematical Institute \\ Leiden University  \\
Leiden, Netherlands}

\email{ajavanp@math.leidenuniv.nl}


\begin{abstract}
We  explicitly bound  the Faltings height of a curve over $\Qbar$ polynomially in its Belyi degree.
Similar bounds are proven for three other Arakelov invariants: the discriminant,
Faltings' delta invariant and the self-intersection of the dualizing sheaf.
Our results allow us to explicitly bound these Arakelov invariants for modular curves,
Hurwitz curves and Fermat curves in terms of their genus. Moreover,
as an application, we show that the Couveignes-Edixhoven-Bruin algorithm to compute
coefficients of  modular forms for  congruence subgroups of $\mathrm{SL}_2(\Z)$ runs in
polynomial time under the Riemann hypothesis for $\zeta$-functions of  number fields. This was known before only for certain congruence subgroups.
Finally, we use our results to prove a conjecture of Edixhoven, de Jong and Schepers on the
Faltings height of a cover of $\p^1_{\Z}$ with fixed branch locus.  
\end{abstract}

\maketitle

\fi

\ifamslatex \maketitle \fi


\section{Introduction and statement of results}\label{intro} We prove that stable Arakelov invariants of a curve over a number field are polynomial in the Belyi degree. We apply our results to give algorithmic, geometric and Diophantine applications.

\subsection{Bounds for Arakelov invariants of three-point covers} Let $\Qbar$ be an algebraic closure of the field of rational numbers $\mathbb{Q}$.
 Let $X$ be a smooth projective connected curve over $\Qbar$ of genus~$g$. Belyi \cite{Belyi} proved that there exists a finite morphism $X\to \p^1_{\Qbar}$ ramified over at most three points. Let $\deg_B(X)$ denote the Belyi degree of $X$, i.e., the minimal degree of a finite morphism $X\to \p^1_{\Qbar}$ unramified over $\p^1_{\Qbar}\backslash\{0,1,\infty\}$. Since the topological fundamental group of the projective line $\p^1(\C)$ minus three points is finitely generated, the set of $\Qbar$-isomorphism classes of curves with bounded Belyi degree is finite.

We prove that, if $g\geq 1$,
the  Faltings height $h_{\Fal}(X)$,
the Faltings delta invariant $\delta_{\Fal}(X)$,
the discriminant $\Delta(X)$ and the self-intersection of
the dualizing sheaf $e(X)$ are bounded by a polynomial in $\deg_B(X)$; the precise definitions of these Arakelov invariants of $X$ are given in Section \ref{invariants}.


\begin{thm}\label{mainthm} For any smooth projective connected curve $X$ over $\Qbar$ of genus $g\geq 1$,
\[  \begin{array}{ccccc} -\log(2\pi) g & \leq &   h_{\Fal}(X) & \leq  & 13\cdot 10^6  g\deg_{B}(X)^5\\
0 & \leq &  e(X) &  \leq  &  3\cdot 10^7 (g-1) \deg_{B}(X)^5 \\
0 & \leq & \Delta(X) &   \leq & 5\cdot 10^8 g^2 \deg_{B}(X)^5 \\
-10^8 g^2 \deg_{B}(X)^5 & \leq & \delta_{\Fal}(X)  & \leq &  2\cdot 10^8 g\deg_{B}(X)^5. \end{array} \] 
\end{thm}

The Arakelov invariants in Theorem \ref{mainthm} all have a different flavour to them.
For example, the Faltings height $h_{\Fal}(X)$ plays a key role in Faltings' proof of his finiteness theorem on abelian varieties; see \cite{Faltings2}. On the other hand, the strict positivity of $e(X)$ (when $g\geq 2$) is related to the Bogomolov conjecture; see \cite{Szpiro7}.
The discriminant $\Delta(X)$ ``measures'' the bad reduction of the curve $X/\Qbar$, and appears in Szpiro's discriminant conjecture for semi-stable elliptic curves; see \cite{Szpiro6}. Finally, as was remarked by Faltings in his introduction to \cite{Faltings1},  Faltings' delta invariant $\delta_{\Fal}(X)$ can be viewed as the minus logarithm of a ``distance''
to the boundary of the moduli space of compact connected Riemann surfaces of genus~$g$. 

We were first led to investigate this problem by work of Edixhoven, de Jong and Schepers on covers of complex algebraic surfaces with fixed branch locus; see \cite{EdJoSc}.
They conjectured an arithmetic analogue (\cite[Conjecture 5.1]{EdJoSc}) of their main theorem (Theorem 1.1 in \emph{loc. cit.}). We use our results to prove this conjecture; see Section \ref{conjecture} for a more precise statement.

\subsection{Outline of proof}
To prove Theorem \ref{mainthm} we will use Arakelov theory for curves over a number field $K$. To apply Arakelov theory in this context, we will work with \textit{arithmetic surfaces} associated to such curves, i.e., regular projective models over the ring of integers $O_K$ of $K$.  We refer the reader to Section \ref{arakelovs} for precise definitions and basic properties of Arakelov's intersection pairing on an arithmetic surface. Then, for any smooth projective connected curve $X$ over $\Qbar$ of genus $g\geq 1$, we define the Faltings height $h_{\Fal}(X)$, the discriminant $\Delta(X)$, Faltings' delta invariant $\delta_{\Fal}(X)$ and the self-intersection of the dualizing sheaf $e(X)$ in Section \ref{invariants}. These are the four Arakelov invariants appearing in Theorem \ref{mainthm}.

We introduce two functions on $X(\Qbar)$ in Section \ref{invariants}: the canonical Arakelov height function and the Arakelov norm of the Wronskian differential.  We show that, to prove Theorem \ref{mainthm}, it suffices to bound the canonical height of some non-Weierstrass point and the Arakelov norm of the Wronskian differential at this point; see Theorem \ref{upperboundinv} for a precise statement. 

We estimate Arakelov-Green functions and Arakelov norms of Wronskian differentials on finite \'etale covers of the modular curve $Y(2)$ in Theorem \ref{MerklResult} and Proposition \ref{Wronskian2}, respectively. In our proof we use an explicit version of a result of Merkl on the Arakelov-Green function; see Theorem \ref{Merkl}. This version of Merkl's theorem was obtained by Peter Bruin in his master's thesis. The proof of this version of Merkl's theorem is reproduced in the appendix by Peter Bruin.

In Section \ref{belyiheights} we prove the existence of a non-Weierstrass point on $X$ of bounded height; see Theorem \ref{heightboundlast}. The proof of Theorem \ref{heightboundlast} relies on our bounds for Arakelov-Green functions (Theorem \ref{MerklResult}), the existence of a ``wild'' model (Theorem \ref{model}) and Lenstra's generalization of Dedekind's discriminant conjecture for discrete valuation rings of characteristic 0 (Proposition \ref{different0}).

A precise combination of the above results constitutes the proof of Theorem \ref{mainthm} given in Section \ref{proofofmaintheorem}.

\subsection{Arakelov invariants of covers of curves with fixed branch locus}\label{coversofcurves} We apply Theorem \ref{mainthm} to prove explicit bounds for the height of a cover of curves. Let us be more precise.

For any finite subset $B\subset \p^1(\Qbar)$ and integer $d\geq 1$, the set of smooth projective connected curves $X$ over $\Qbar$ such that there exists a finite morphism $X\to \p^1_{\Qbar}$ \'etale over $\p^1_{\Qbar}-B$ of degree $d$ is finite. In particular,  the Faltings height of $X$ is bounded by a real number depending only on $B$ and $d$. In this section we give an explicit version of this statement.  To state our result we need to define the height of $B$.

For any finite set $B\subset \p^1(\Qbar)$, define the (exponential) height as $H_B= \max \{ H(\alpha): \alpha \in B\}$, where the height  $H(\alpha)$ of an element $\alpha$ in $\Qbar$ is defined  as $H(\alpha) = \left( \prod_v \max(1,\Vert \alpha\Vert_v) \right)^{1/[K:\Q]}$.
Here $K$ is a number field containing $\alpha$ and the product runs over the set of normalized valuations $v$ of $K$. (As in \cite[Section 2]{Khadjavi} we require our normalization to be such that the product formula holds.)

\begin{thm}\label{mainthmintro}
Let $U$ be a non-empty open subscheme in $\p^1_{\Qbar}$ with complement $B\subset \p^1(\Qbar)$. Let $N$ be the number of elements in the orbit of $B$ under the action of $\mathrm{Gal}(\Qbar/\Q)$. Then, for any finite morphism $\pi:Y\to \p^1_{\Qbar}$ \'etale over $U$, where $Y$ is a smooth projective connected curve over $\Qbar$ of genus $g\geq 1$, 
\[ \begin{array}{ccccc} -\log(2\pi)g & \leq &  h_{\Fal}(Y) & \leq  & 13\cdot 10^6 g(4NH_B)^{45N^3 2^{N-2}N!}(\deg \pi)^5  \\
0 &\leq &  e(Y)  & \leq  & 3\cdot 10^7(g-1)(4NH_B)^{45N^3 2^{N-2}N!}(\deg \pi)^5 \\
 0 &\leq &  \Delta(Y)  & \leq &  5\cdot 10^8  g^2(4NH_B)^{45N^3 2^{N-2}N!}(\deg \pi)^5 \\
- 10^8 g^2 (4NH_B)^{45N^3 2^{N-2}N!}(\deg \pi)^5  & \leq &  \delta_{\Fal}(Y) & \leq &  2\cdot 10^8 g (4NH_B)^{45 N^3 2^{N-2}N!} (\deg \pi)^5.
\end{array} \] 
\end{thm}

Theorem \ref{mainthmintro} is a consequence of Theorem \ref{mainthm2}. Note that in Theorem \ref{mainthm2} we consider branched covers of any  curve over $\Qbar$ (i.e., not only $\p^1_{\Qbar}$). We use Theorem \ref{mainthmintro} to prove  \cite[Conjecture 5.1]{EdJoSc}. 

\subsection{Diophantine application}

Explicit bounds for  Arakelov invariants of curves of genus $g\geq 2$ over a number field $K$ and with bad reduction outside a finite set $S$ of finite places of $K$ imply famous conjectures in Diophantine geometry such as the \textit{effective Mordell conjecture} and the \textit{effective Shafarevich conjecture}; see \cite{Remo} and \cite{Szpiro1}. We note that Theorem \ref{mainthm} shows that one ``could'' replace Arakelov invariants by the Belyi degree to prove these conjectures. We use this philosophy to deal with cyclic covers of prime degree. In fact, in \cite{JvK}, joint with von K\"anel, we utilize Theorem \ref{mainthm} and the theory of logarithmic forms to prove  Szpiro's small points conjecture (\cite[p. 284]{Szpiro3} and \cite{Szpiro4}) for curves that are cyclic covers of the projective line of prime degree; see \cite[Theorem 3.1]{JvK} for a precise statement. In particular, we prove Szpiro's small points conjecture for hyperelliptic curves.

\subsection{Modular curves, Fermat curves, Hurwitz curves and Galois Belyi curves}
Let $X$ be a smooth projective connected curve over $\Qbar$ of genus $g\geq 2$. We say that $X$ is a Fermat curve if there exists an integer $n\geq 4$ such that $X$ is isomorphic to the planar curve $\{x^n+y^n =z^n\}$. Moreover, we say that $X$ is a Hurwitz curve if $\#\mathrm{Aut}(X) = 84(g-1)$. Also, we say that $X$ is a Galois Belyi curve if  the quotient $X/\mathrm{Aut}(X)$ is isomorphic to $\p^1_{\Qbar}$ and the morphism $X\to X/\mathrm{Aut}(X)$ is ramified over exactly three points; see \cite[Proposition 2.4]{ClVo}, \cite{Wolfart1} or \cite{Wolfart2}. Note that Fermat curves and Hurwitz curves are Galois Belyi curves. Finally, we say that $X$ is a modular curve if $X_\C$ is a classical congruence modular curve with respect to some (hence any) embedding $\Qbar\to \C$. 

If $X$ is a Galois Belyi curve, we have $\deg_B(X) \leq 84(g-1)$. In \cite{Zograf} Zograf proved that,  if $X$ is a modular curve, then $\deg_B(X) \leq 128(g+1)$. Combining these  bounds with Theorem \ref{mainthm} we obtain the following corollary.

\begin{cor}\label{modferwol}
Let $X$ be a smooth projective connected curve over $\Qbar$ of genus $g\geq 1$. Suppose that $X$ is a modular curve or Galois Belyi curve. Then \[ \max(  h_{\Fal}(X),e(X),\Delta(X), \vert \delta_{\Fal}(X)\vert) \leq 2\cdot 10^{19} g^2(g+1)^5 .\]
\end{cor}

\begin{opm} 
Let $\Gamma \subset \mathrm{SL}_2(\Z)$ be   a finite index subgroup, and let $X$ be the compactification of $\Gamma\backslash \h$ obtained by adding the cusps, where $\Gamma$ acts on the complex upper half-plane $\h$ via M\"obius transformations. Let $X(1)$ denote the compactification of  $\mathrm{SL}_2(\Z)\bs \h$. The inclusion $\Gamma\subset \mathrm{SL}_2(\Z)$ induces a morphism $X\to X(1)$. For $\Qbar\subset \C$ an embedding, there is a unique finite morphism $Y\to \p^1_{\Qbar}$  of smooth projective connected curves over $\Qbar$ corresponding to $X\longrightarrow X(1)$. The Belyi degree of $Y$ is bounded from above  by the index $d$ of $\Gamma$ in $\mathrm{SL}_2(\Z)$.  In particular,  \[ \max(  h_{\Fal}(Y),e(Y),\Delta(Y), \vert \delta_{\Fal}(Y)\vert) \leq 10^{9} d^7. \]
\end{opm}

\begin{opm}
Non-explicit versions of Corollary \ref{modferwol} were previously known for certain modular curves. Firstly, polynomial bounds for Arakelov invariants of $X_0(n)$ with $n$ squarefree were previously known; see \cite[Th\'eor\`eme 1.1]{Ullmo}, \cite[Corollaire 1.3]{Ullmo}, \cite{AbUl}, \cite[Th\'eor\`eme 1.1]{MicUll} and \cite{JorKra2}. The proofs of these results rely on the theory of modular curves. Also, similar results for Arakelov invariants of $X_1(n)$ with $n$ squarefree were shown in  \cite{EdJo3} and \cite{Mayer}. Finally, bounds for the self-intersection of the dualizing sheaf of a Fermat curve of prime exponent are given in \cite{CuKu} and \cite{Ku}.
\end{opm}

\subsection{The Couveignes-Edixhoven-Bruin algorithm}

Corollary \ref{modferwol} guarantees that, under the Riemann hypothesis for $\zeta$-functions of  number fields, the Couveignes-Edixhoven-Bruin algorithm to compute coefficients of modular forms runs in polynomial time; see Theorem \ref{CoEdBr} for a more precise statement.

\subsection*{Conventions} By $\log$ we mean the principal value of the natural logarithm. Finally, we define the maximum of the empty set and the product taken over the empty set as 1.

\subsection*{Acknowledgements}  I would like to thank Peter Bruin,  Bas Edixhoven and  Robin de Jong. They introduced  us to Arakelov theory and Merkl's theorem, and I am grateful to them for many inspiring discussions and their help in writing this article. Also, I would like to thank Rafael von K\"anel and Jan Steffen M\"uller for motivating discussions about this article. I would like to thank Jean-Beno\^it Bost and Gerard Freixas for discussions on Arakelov geometry, Yuri Bilu for inspiring discussions, J\"urg  Kramer for  discussions on Faltings' delta invariant, Hendrik Lenstra and Bart de Smit for their help in proving Proposition \ref{different0},  Qing Liu for answering our questions on models of finite morphisms of curves and Karl Schwede for helpful discussions about the geometry of surfaces. 

\section{Arakelov geometry of curves over number fields}

We are going to apply Arakelov theory to smooth projective geometrically connected curves~$X$ over number fields~$K$. In~\cite{Arakelov} Arakelov defined an intersection theory on the \emph{arithmetic surfaces} attached to such curves.  In~\cite{Faltings1}  Faltings extended Arakelov's work. In this section we aim at giving the necessary definitions and results for what we need later (and we need at least to fix our notation).

We start with some preparations concerning Riemann surfaces and arithmetic surfaces. In Section \ref{invariants} we define the (stable) Arakelov invariants of $X$ appearing in Theorem \ref{mainthm}. Finally, we prove bounds for  Arakelov invariants of $X$ in the height and the  Arakelov norm of the Wronskian differential of a non-Weierstrass point; see Theorem \ref{upperboundinv}.

\subsection{Arakelov invariants of Riemann surfaces} \label{admissible}
Let $X$ be a compact connected Riemann surface of genus $g\geq 1$. The space of holomorphic differentials $\rH^0(X,\Omega_X^1)$ carries a natural hermitian inner product:
\begin{eqnarray*}\label{eqn_nat_inner_pro} (\omega,\eta) &\mapsto& \frac{i}{2} \int_X \omega \wedge \ol{\eta}. \end{eqnarray*} For any orthonormal basis $(\omega_1,\ldots,\omega_g)$ with respect to this inner product, the Arakelov $(1,1)$-form is the smooth positive real-valued $(1,1)$-form $\mu$ on~$X$ given by $\mu =\frac{i}{2g} \sum_{k=1}^g \omega_k \wedge \ol{\omega_k}$. Note that $\mu$ is independent of the choice of orthonormal basis. Moreover, $\int_X \mu=1$.

Let $\gr_X$ be the Arakelov-Green function on $(X\times X)\backslash \Delta$, where $\Delta \subset X\times X$ denotes the diagonal; see \cite{Arakelov}, \cite{deJo},  \cite{EdJo1} or \cite{Faltings1}. The Arakelov-Green functions determine certain metrics whose curvature forms are multiples of $\mu$, called \textit{admissible metrics}, on all line bundles~$\mathcal{O}_X(D)$, where $D$ is a divisor on~$X$, as well as on the holomorphic cotangent bundle~$\Omega^1_X$. Explicitly: for $D=\sum_P D_P P$ a divisor on~$X$, the metric $\| {\cdot}\|$ on $\mathcal{O}_X(D)$ satisfies $\log\|1\|(Q) = \gr_X(D,Q)$  for all $Q$ away from the support of~$D$, where $\gr_X(D,Q) := \sum_P n_P \gr_X(P,Q)$. Furthermore, for a local coordinate $z$ at a point $a$ in $X$, the metric $\Vert \cdot \Vert_{\Ar}$ on the sheaf $\Omega^1_{X}$ satisfies \[ -\log \Vert dz \Vert_{\mathrm{Ar}}(a) = \lim_{b\to a}\left( \gr_{X}(a,b) - \log \vert z(a) - z(b) \vert  \right). \] We will work with these metrics on~$\mathcal{O}_X(P)$ and $\Omega_X^1$ (as well as on tensor product combinations of them) and refer to them as \textit{Arakelov metrics}.  A metrised line bundle $\mathcal{L}$  is called \textit{admissible} if, up to a constant scaling factor, it is isomorphic to one of the admissible
bundles~$\mathcal{O}_X(D)$.  The line bundle $\Omega^1_X$ endowed with the above metric is admissible; see \cite{Arakelov}.

For any admissible line bundle~$\mathcal{L}$, we endow the determinant of cohomology \[\lambda(\mathcal{L}) = \det \rH^0(X,\mathcal{L}) \otimes \det \rH^1(X,\mathcal{L})^\vee\] of the underlying line bundle with the  Faltings metric;  see \cite[Theorem 1]{Faltings1}. We normalize this metric so that the metric on $\lambda(\Omega^1_X) =\det \rH^0(X,\Omega^1_X)$ is induced by the hermitian inner product on~$\rH^0(X,\Omega_X^1)$ given above.  

Let $\HH_g$ be the Siegel upper half space of complex symmetric $g$-by-$g$-matrices with positive definite imaginary part. Let $\tau$ in~$\HH_g$ be the period matrix attached to a symplectic basis of $\rH_1(X,\Z)$ and consider the analytic Jacobian $J_\tau(X) = \C^g /(\Z^g + \tau \Z^g)$ attached to~$\tau$. On $\C^g$ one has a theta function $\vartheta(z;\tau)=\vartheta_{0,0}(z;\tau) = \sum_{n\in\Z^g} \exp(\pi i\,{}^t\hspace{-0.1em}n \tau n + 2\pi i\, {}^t\hspace{-0.1em}n z)$,  giving rise to a reduced effective divisor~$\Theta_0$ and a line bundle
$\mathcal{O}(\Theta_0)$ on~$J_\tau(X)$. The function $\vartheta$ is not well-defined on ~$J_\tau(X)$.  Instead, we consider the function
\begin{eqnarray}\label{eqn_thetanorm}
\|\vartheta\|(z;\tau) &=&
(\det \Im(\tau))^{1/4} \exp(-\pi\,{}^t\hspace{-0.1em}y
(\Im(\tau))^{-1} y)|\vartheta(z;\tau)|,
\end{eqnarray}
with $y = \Im(z)$. One can check that $\|\vartheta\|$ descends to a function on~$J_\tau(X)$. Now consider on the other hand the set $\mathrm{Pic}_{g-1}(X)$ of divisor classes of degree $g-1$ on~$X$. It comes with a canonical subset $\Theta$ given by the classes of effective divisors and a canonical bijection $\mathrm{Pic}_{g-1}(X)\isomlto J_\tau(X)$ mapping $\Theta$ onto~$\Theta_0$. As a result, we can equip $\mathrm{Pic}_{g-1}(X)$ with the structure of a compact complex manifold, together with a divisor $\Theta$ and a line bundle~$\mathcal{O}(\Theta)$.  Note that we obtain $\|\vartheta\|$ as a function on~$\mathrm{Pic}_{g-1}(X)$. It can be checked that this function is independent of the choice of~$\tau$. Furthermore, note that $\|\vartheta\|$ gives a canonical way to put a metric on the line bundle $\mathcal{O}(\Theta)$ on~$\mathrm{Pic}_{g-1}(X)$. 

For any line bundle $\mathcal{L}$ of degree~$g-1$ there is a canonical isomorphism from $\lambda(\mathcal L)$ to $\mathcal{O}(-\Theta)[\mathcal{L}]$, the fibre of $\mathcal{O}(-\Theta)$ at the point $[\mathcal{L}]$ in $\mathrm{Pic}_{g-1}(X)$ determined by~$\mathcal{L}$.  Faltings proves that when we give both sides the metrics discussed above, the norm of this isomorphism is a constant independent of~$\mathcal{L}$; see \cite[Section 3]{Faltings1}. We will write this norm as $\exp(\delta_{\Fal}(X)/8)$ and refer to $\delta_{\Fal}(X)$ as Faltings' delta invariant of $X$. 

Let $S(X)$ be the invariant  of $X$ defined in \cite[Definition 2.2]{deJo}. More explicitly, by \cite[Theorem 2.5]{deJo},
\begin{eqnarray}\label{SX} \log S(X) &=& -\int_X \log \| \vartheta \|  (gP-Q) \cdot \mu(P), \end{eqnarray} where $Q$ is any point on $X$. It is related to Faltings' delta invariant $\delta_{\Fal}(X)$. In fact, let $(\omega_1,\ldots,\omega_g)$ be
an orthonormal basis of $\rH^0(X,\Omega_X^1)$. Let $b$ be a point on $X$ and let $z$ be a local
coordinate about $b$. Write $\omega_k = f_k dz$ for $k=1,\ldots,g$. We have a holomorphic function
\[W_z(\omega) = \det\left( \frac{1}{(l-1)!} \frac{d^{l-1}f_k}{dz^{l-1}}\right)_{1\leq k,l\leq g}\]
locally about $b$ from which we build the $g(g+1)/2$-fold holomorphic differential $  W_z(\omega) (dz)^{\otimes g(g+1)/2}$.
It is readily checked that this holomorphic differential is independent of the choice of local coordinate and orthonormal basis.
Thus, the holomorphic differential $  W_z(\omega) (dz)^{\otimes g(g+1)/2}$ extends over $X$ to give a non-zero global section, denoted by $\mathrm{Wr}$, of the line bundle $\Omega^{\otimes g(g+1)/2}_{X}$.   The divisor of the non-zero global section $\mathrm{Wr}$, denoted by $\mathcal{W}$, is the divisor of Weierstrass points. This divisor is effective of degree $g^3-g$.   We follow  \cite[Definition 5.3]{deJo} and denote the constant norm of the canonical isomorphism of (abstract) line bundles \[\Omega_X^{g(g+1)/2} \otimes_{\mathcal{O}_X}\left( \Lambda^g \rH^0(X,\Omega^1_X) \otimes_{\C} \mathcal{O}_X \right)^{\vee}\longrightarrow \mathcal{O}_X(\mathcal{W}) \] by $R(X)$. Then, \begin{eqnarray}\label{Sinvariant} \log S(X) & =&   \frac{1}{8}\delta_{\Fal}(X) + \log R(X). \end{eqnarray} Moreover, for any non-Weierstrass point $b$ in $X$,\begin{eqnarray}\label{Wronskian}  \gr_X(\mathcal W,b) - \log R(X) &=& \log \Vert \mathrm{Wr}\Vert_{\Ar}(b).\end{eqnarray}

\subsection{Arakelov's intersection pairing on an arithmetic surface} \label{arakelovs}
Let $K$ be a number field with ring of integers $O_K$, and let $S=\Spec O_K$. Let $p:\mcX\to S$ be an arithmetic surface, i.e., an integral regular flat projective $S$-scheme of relative dimension 1 with geometrically connected fibres. For the sake of clarity, let us note that $p:\mcX\to S$ is a regular projective model of the generic fibre $\mcX_K \to \Spec K$ in the sense of \cite[Definition~10.1.1]{Liu2}.

In this section, we will assume the genus of the generic fibre $\mcX_K$ to be positive. An Arakelov divisor $D$ on $\mcX$ is a divisor $D_{\fin}$ on $\mcX$, plus a contribution $D_{\inff} = \sum_\sigma \alpha_{\sigma} F_\sigma$ running over the embeddings $\sigma:K\longrightarrow \C$ of $K$ into the complex numbers. Here the $\alpha_\sigma$ are real numbers and the $F_\sigma$ are formally the ``fibers at infinity'', corresponding to the Riemann surfaces $\mcX_\sigma$ associated to the algebraic curves $\mcX\times_{O_K,\sigma} \C$. We let  $\widehat{\Div}(\mcX)$ denote the group of Arakelov divisors on $\mcX$. To a non-zero rational function $f$ on $\mcX$, we associate an Arakelov divisor $\widehat{\divv}(f) := (f)_{\fin} + (f)_{\inff}$ with $(f)_{\fin}$ the usual divisor associated to $f$ on $\mcX$, and $(f)_{\inff} = \sum_\sigma v_\sigma(f) F_\sigma$, where $v_\sigma(f):= -\int_{\mcX_\sigma} \log\vert f\vert_\sigma \cdot \mu_\sigma$. Here $\mu_\sigma$ is the Arakelov $(1,1)$-form on $\mcX_\sigma$. We will say that two Arakelov divisors on $\mcX$ are linearly equivalent if their difference is of the form $\widehat{\divv}(f)$ for some non-zero rational function $f$ on $\mcX$. We let $\widehat{\Cl}(\mcX)$ denote the group of Arakelov divisors modulo linear equivalence on $\mcX$.

In \cite{Arakelov} Arakelov showed that there exists a  unique symmetric bilinear map  $(\cdot, \cdot):\widehat{\Cl}(\mcX)\times \widehat{\Cl}(\mcX)\longrightarrow \R$ with the following properties:
\begin{itemize}
\item if $D$ and $E$ are effective divisors on $\mcX$ without common component,  then \[(D,E) = (D,E)_{\fin} -\sum_{\sigma:K\to \C}  \gr_{\mcX_\sigma}(D_\sigma,E_\sigma), \] where $\sigma$ runs over the complex embeddings of $K$. Here $(D,E)_{\fin}$ denotes the usual intersection number of $D$ and $E$ as in \cite[Section~9.1]{Liu2}, i.e., \[(D,E)_{\fin} = \sum_{s \in \vert S\vert} i_s(D,E) \log \# k(s),\] where $s$ runs over the set $\vert S \vert$ of closed points of $S$, $i_s(D,E)$ is the intersection multiplicity of $D$ and $E$ at $s$ and $k(s)$ denotes the residue field of $s$. Note that if $D$ or $E$ is vertical, the sum $\sum_{\sigma:K\to \C} \gr_{\mcX_\sigma}(D_\sigma,E_\sigma)$ is zero;
\item  if $D$ is a horizontal divisor of generic degree $n$ over $S$, then $(D,F_\sigma) = n$ for every $\sigma:K\longrightarrow \C$;
\item if $\sigma_1,\sigma_2:K\to \C$ are complex embeddings, then $(F_{\sigma_1}, F_{\sigma_2}) =0 $.
\end{itemize}
An \textit{admissible line bundle} on $\mcX$ is the datum of a line bundle $\mathcal{L}$ on $\mcX$, together with admissible metrics on the restrictions $\mathcal{L}_\sigma$ of $\mathcal{L}$ to the  $\mcX_\sigma$. Let $\widehat{\Pic}(\mcX)$ denote the group of isomorphism classes of admissible line bundles on $\mcX$.   To any Arakelov divisor $D= D_{\fin} + D_{\inff}$ with $D_{\inff} = \sum_{\sigma} \alpha_\sigma F_\sigma$, we can associate an admissible line bundle $\mathcal{O}_{\mcX}(D)$. In fact, for the underlying line bundle of $\mathcal{O}_{\mcX}(D)$ we take $\mathcal{O}_{\mcX}(D_{\fin})$. Then, we make this into an admissible line bundle by equipping the pull-back of $\mathcal{O}_{\mcX}(D_{\fin})$ to each $\mcX_\sigma$ with its Arakelov metric, multiplied by $\exp(-\alpha_\sigma)$. This induces an isomorphism \[\xymatrix{\widehat{\Cl}(\mcX)\ar[r]^{\sim} &\widehat{\Pic}(\mcX).}\] In particular, the Arakelov intersection of two admissible line bundles on $\mcX$ is well-defined.

Recall that a metrised line bundle $(\mathcal{L},\|{\cdot}\|)$ on $\Spec O_K$ corresponds to an invertible $O_K$-module, $L$, say, with hermitian metrics on the $L_\sigma:=\C\otimes_{\sigma,O_K}L$. The \emph{Arakelov degree} of~$(\mathcal{L},\|{\cdot}\|)$ is the real number defined by:
\begin{eqnarray*}\label{eqn_ar_degree}
\widehat{\deg}(\mathcal{L})= \widehat{\deg}(\mathcal{L},\|{\cdot}\|) =
\log\#(L/O_Ks) -\sum_{\sigma\colon K\to\C}\log\|s\|_\sigma,
\end{eqnarray*}
where $s$ is any non-zero element of~$L$ (independence of the choice of~$s$ follows from the product formula). 

Note that the relative dualizing sheaf $\omega_{\mcX/O_K}$ of $p:\mcX \to S$ is an admissible line bundle on $\mcX$ if we endow the restrictions $\Omega^1_{\mcX_\sigma}$ of $\omega_{\mcX/O_K}$ to the $\mcX_\sigma$ with their Arakelov metric. Furthermore, for any section $P:S\to \mcX$, we have \[\widehat{\deg} P^\ast \omega_{\mcX/O_K} = (\mathcal{O}_X(P), \omega_{\mcX/O_K}) =: (P,\omega_{\mcX/O_K}),\] where we endow the line bundle $P^\ast \omega_{\mcX/O_K}$ on $\Spec O_K$ with the pull-back metric.

\begin{defn}\label{semi-stable}
 We say that $\mcX$ is \textit{semi-stable  (or nodal) over $S$} if every geometric fibre of $\mcX$ over $S$ is reduced and has only ordinary double singularities; see \cite[Definition~10.3.1]{Liu2}. We say that $\mcX$ is \textit{(relatively) minimal} if it does not contain any exceptional divisor; see \cite[Definition~9.3.12]{Liu2}.
\end{defn}

\begin{opm}
Suppose that $\mcX$ is semi-stable over $S$ and minimal. The blowing-up $\mcY\to\mcX$ along a smooth closed point on $\mcX$ is semi-stable over $S$, but no longer minimal.
\end{opm}

\subsection{Arakelov invariants of curves }\label{invariants}
Let $X$ be a smooth projective connected curve over $\Qbar$ of genus $g\geq 1$. Let $K$ be a number field such that $X$ has a semi-stable minimal regular model $p:\mcX\to \Spec O_K$; see Theorems 10.1.8, 10.3.34.a and 10.4.3 in \cite{Liu2}. (Note that we implicitly chose an embedding $K\to \Qbar$.)

The \textit{Faltings delta invariant} of $X$, denoted by $\delta_{\Fal}(X)$, is defined as \[\delta_{\Fal}(X) =\frac{1}{[K:\Q]}\sum_{\sigma:K\to \C} \delta_{\Fal}(\mcX_\sigma),\] where $\sigma$ runs over the complex embeddings of $K$ into $\mathbb{C}$. Similarly, we define \[ \Vert \vartheta \Vert_{\textrm{max}}(X) = \left(\prod_{\sigma:K\to \C} \max_{\mathrm{Pic}_{g-1}(\mcX_\sigma)}\Vert \vartheta\Vert\right)^{1/[K:\Q]}.\] Moreover, we define \[R(X) = \left(\prod_{\sigma:K\to \C} R(\mcX_\sigma)\right)^{1/[K:\Q]}, \quad S(X) = \left(\prod_{\sigma:K\to \C} S(\mcX_\sigma)\right)^{1/[K:\Q]}.\]
The \emph{Faltings height} of $X$ is defined by \[h_{\Fal}(X) = \frac{\widehat{\deg} \det p_\ast \omega_{\mathcal{X}/O_K}}{[K:\Q]} = \frac{\widehat{\deg} \det R^\cdot p_\ast \mathcal{O}_{\mathcal{X}}}{[K:\Q]},\] where we endow the determinant of cohomology with the Faltings metric; see Section \ref{admissible}. Note that $h_{\Fal}(X)$ coincides with the stable Faltings height of the Jacobian of $\mcX_K$; see  \cite[Lemme~3.2.1, Chapitre~I]{Szpiroa}. Furthermore, we define the \textit{self-intersection of the dualizing sheaf} of $X$, denoted by $e(X)$, as \[e(X):= \frac{(\omega_{\mathcal{X}/O_K},\omega_{\mathcal{X}/O_K})}{[K:\Q]},\]
where we use Arakelov's intersection pairing on the arithmetic surface $\mcX/O_K$.  The \textit{discriminant} of $X$, denoted by $\Delta(X)$, is defined as \[\Delta(X) = \frac{\sum_{\mathfrak{p}\subset O_K} \delta_{\mathfrak{p}} \log\# k(\mathfrak{p})}{[K:\Q]},\] where $\mathfrak{p}$ runs through the maximal ideals of $O_K$ and $\delta_{\mathfrak{p}}$ denotes the number of singularities in the geometric fibre of $p:\mcX\to \Spec O_K$ over $\mathfrak{p}$. These invariants of $X$ are well-defined; see \cite[Section 5.4]{Moret-Bailly3}.

To bound the above Arakelov invariants, we introduce two functions on $X(\Qbar)$: the height and the Arakelov norm of the Wronskian differential.
More precisely, let $b\in X(\Qbar)$ and suppose that $b$ induces a section $P$ of $\mcX$ over $O_K$.
Then we define the \textit{height of $b$}, denoted by $h(b)$, to be \[h(b) = \frac{\widehat{\deg}P^\ast \omega_{\mcX/O_K}}{[K:\Q]} = \frac{(P,\omega_{\mcX/O_K})}{[K:\Q]}.\]
Note that the height of $b$ is the stable canonical height of a point, in the Arakelov-theoretic sense, with respect to the admissible line bundle $\omega_{\mcX/O_K}$.   We define the Arakelov norm of the Wronskian differential at $b$ as \[\Vert \mathrm{Wr}\Vert_{\Ar}(b) = \left(\prod_{\sigma:K\to \C} \Vert \mathrm{Wr}\Vert_{\Ar}(b_\sigma)\right)^{1/[K:\Q]}. \] These functions  on $X(\Qbar)$ are well-defined; see \cite[Section 5.4]{Moret-Bailly3}.

Changing the model for $X$ might change the height of a point. Let us show that the height of a point does not become smaller if we take another regular model over $O_K$.
\begin{lem}\label{heightbigger}
Let $\mcY\to \Spec O_K$ be an arithmetic surface. Assume that $\mcY$ is a model for $\mcX_K$. If $Q$ denotes the section of $\mcY$ over $O_K$ induced by $b\in X(\Qbar)$, then \[h(b) \leq \frac{(Q,\omega_{\mcY/O_K})}{[K:\Q]}.\]
\end{lem}
\begin{proof}
By the minimality of $\mcX$, there is a unique birational morphism $\phi:\mcY\to \mcX$; see \cite[Corollary~9.3.24]{Liu2}. By the factorization theorem, this morphism is made up of a finite sequence \[\xymatrix{\mcY = \mcY_n \ar[r]^{\phi_n} & \mcY_{n-1} \ar[r]^{\phi_{n-1}} & \ldots \ar[r]^{\phi_1} & \mathcal{Y}_0 = \mcX}\] of blowing-ups along closed points; see \cite[Theorem~9.2.2]{Liu2}.  For $i=1,\ldots,n$, let  $E_i \subset \mcY_i$ denote the exceptional divisor of $\phi_i$. Since the line bundles $\omega_{\mcY_i/O_K}$ and $\phi^\ast_i\omega_{\mcY_{i-1}/O_K}$ agree on $\mcY_i - E_i$, there is an integer $a$ such that \[\omega_{\mcY_i/O_K} = \phi_i^\ast \omega_{\mcY_{i-1}/O_K}\otimes_{\mathcal{O}_{\mcY_i}} \mathcal{O}_{\mcY_i}(aE_i).\] Applying the adjunction formula,  we see that $a=1$. Since $\phi_i$ restricts to the identity morphism on the generic fibre, we have a canonical isomorphism of admissible line bundles  \[\omega_{\mcY_i/O_K} = \phi_i^\ast \omega_{\mcY_{i-1}/O_K}\otimes_{\mathcal{O}_{\mcY_i}} \mathcal{O}_{\mcY_i}(E_i).\] Let  $Q_i$ denote the section of $\mcY_i$ over $O_K$ induced by $b\in X(\Qbar)$. Then  \[(Q_i,\omega_{\mcY_i/O_K}) = (Q_{i},\phi^\ast_i \omega_{\mcY_{i-1}/O_K}) + (Q_i,E_i) \geq (Q_i, \phi^\ast_i\omega_{\mcY_{i-1}/O_K}) = (Q_{i-1},\omega_{\mcY_{i-1}/O_K}),\] where we used the projection formula in the last equality. Therefore, we conclude that \[(Q,\omega_{\mcY/O_K}) = (Q_n,\omega_{\mcY_n/O_K}) \geq (Q_0, \omega_{\mcY_0/O_K}) = (P,\omega_{\mcX/O_K})=h(b)[K:\Q]. \qedhere \]
\end{proof}

\subsection{Bounding Arakelov invariants in the height of a non-Weierstrass point}\label{invariants2}
In this section we prove bounds for Arakelov invariants of curves in the height of a non-Weierstrass point and the Arakelov norm of the Wronskian differential in this point.

\begin{thm}\label{upperboundinv} Let $X$ be a smooth projective connected curve over $\Qbar$ of genus $g\geq 1 $. Let $b\in X(\Qbar)$. Then
\[ \begin{array}{ccc} e(X) & \leq & 4g(g-1) h(b),  \\ \delta_{\Fal}(X) & \geq & -90 g^3 - 4g(2g-1)(g+1) h(b) .
 \end{array} \] Suppose that $b$ is not a Weierstrass point. Then
\[ \begin{array}{ccc}
 h_{\Fal}(X) & \leq & \frac{1}{2} g(g+1) h(b) +\log \Vert \mathrm{Wr}\Vert_{\Ar}(b), \\
 \delta_{\Fal}(X) & \leq & 6 g(g+1) h(b) +  12\log \Vert\mathrm{Wr}\Vert_{\Ar}(b)+ 4g\log(2\pi), \\  \Delta(X) &\leq & 2g (g+1)(4g+1) h(b) + 12  \log \Vert \mathrm{Wr}\Vert_{\Ar} (b) + 93g^3.
 \end{array} \]
\end{thm}

This theorem is essential to the proof of Theorem \ref{mainthm} given in Section \ref{heightboundlastsection}. We give a proof of Theorem \ref{upperboundinv} at the end of this section.

\begin{lem}\label{theta}
For a smooth projective connected curve $X$   over $\Qbar$ of genus $g\geq 1$,  \[\log \Vert \vartheta \Vert_{\max}(X) \leq \frac{g}{4} \log \max(1,h_{\Fal}(X)) + (4g^3+5g+1)\log 2.\]
\end{lem}
\begin{proof} We kindly thank R. de Jong for sharing this proof with us. We follow the idea of \cite[Section 2.3.2]{Graf}, see also \cite[Appendice]{Davi}. Let $\mathcal{F}_g$ be the Siegel fundamental domain of dimension  $g$ in the Siegel upper half-space $\mathbb{H}_g$, i.e., the space of complex $(g\times g)$-matrices $\tau$ in $\mathbb{H}_g$ such that the following properties are satisfied. Firstly, for every element $u_{ij}$ of $u=\Re(\tau)$, we have
$\vert u_{ij} \vert \leq 1/2$. Secondly, for every $\gamma$ in $\mathrm{Sp}(2g,\Z)$, we have $\det \Im(\gamma \cdot \tau)\leq \det \Im(\tau)$, and finally, $\Im(\tau)$ is Minkowski-reduced, i.e., for all $\xi = (\xi_1,\ldots,\xi_g) \in \Z^g$ and for all $i$ such that $\xi_i,\ldots,\xi_g$ are non-zero, we have $\xi \Im(\tau)^t \xi \geq (\Im(\tau))_{ii}$ and, for all $1\leq i \leq g-1$ we have $(\Im(\tau))_{i,i+1} \geq 0$. One can show that $\mathcal{F}_g$ contains a representative of each $\mathrm{Sp}(2g,\Z)$-orbit in $\mathbb{H}_g$.

Let $K$ be a number field such that $X$ has a model $X_K$ over $K$. For every  embedding $\sigma:K\to \C$, let $\tau_\sigma$ be an element of $\mathcal{F}_g$ such that $\mathrm{Jac}(X_{K,\sigma}) \cong \C^g/(\tau_\sigma \Z^g+\Z^g)$ as principally polarized abelian varieties, the matrix of the Riemann form induced by the polarization of $\mathrm{Jac}(X_{K,\sigma})$ being $\Im(\tau_\sigma)^{-1}$ on the canonical basis of $\C^g$. By a result of Bost  (see \cite[Lemme 2.12]{Graf} or \cite{Pa}), we have  \begin{eqnarray}\label{yeah} \frac{1}{[K:\Q]} \sum_{\sigma:K\to \C} \log \det(\Im(\tau_\sigma)) & \leq & g \log \max(1,h_{\Fal}(X)) + (2g^3+2)\log(2). \end{eqnarray} Here we used that the Faltings height of $X$ equals the Faltings height of its Jacobian.
Now, let $\vartheta(z;\tau)$ be the Riemann theta function  as in Section \ref{admissible}, where $\tau$ is in $\mathcal{F}_g$ and $z=x+iy$ is in $\C^g$ with $x,y\in \R^g$. Combining (\ref{yeah}) with the upper bound \begin{eqnarray}\label{yeah2} \exp(-\pi^t y (\Im(\tau))^{-1} y) \vert \vartheta (z;\tau)\vert &\leq & 2^{3g^3+5g} \end{eqnarray} implies the result.  Let us prove (\ref{yeah2}). Note that, if we write $y= \Im(z) = (\Im(\tau)) \cdot b$ for $b$ in $\R^g$, \[\exp(-\pi^t g(\Im(\tau))^{-1} y) \vert \vartheta(z;\tau)\vert \leq \sum_{n\in \Z^g} \exp( -\pi^t (n+b) (\Im(\tau))(n+b)).\] Since $\Im(\tau)$ is Minkowski reduced, we have  $^tm\Im(\tau)m \geq c(g) \sum_{i=1}^g m_i^2 (\Im(\tau))_{ii}$ for all $m$ in $\R^g$. Here $c(g) = \left(\frac{4}{g^3}\right)^{g-1} \left(\frac{3}{4}\right)^{g(g-1)/2}$. Also,  $(\Im(\tau))_{ii} \geq \sqrt{3}/2$ for all $i=1,\ldots,g$ (see \cite[Chapter V.4]{Igus} for these facts). We deduce that \begin{eqnarray*} \sum_{n\in \Z^g} \exp(-\pi^t(n+b)(\Im(\tau))(n+b)) &\leq &
\sum_{n\in \Z^g} \exp\left( -\sum_{i=1}^g \pi c(g) (n_i+b_i)^2 (\Im(\tau))_{ii} \right) \\ &\leq & \prod_{i=1}^g \sum_{n_i \in \Z} \exp(-\pi c(g)(n_i + b_i)^2 (\Im(\tau))_{ii}) \\ &\leq &  \prod_{i=1}^g \frac{2}{1-\exp(-\pi c(g) (\Im(\tau))_{ii})} \leq 2^g\left(1+\frac{2}{\pi \sqrt{3} c(g)}\right)^g.
\end{eqnarray*}
This proves (\ref{yeah2}).
\end{proof}

\begin{lem}\label{ineq} Let $a\in \R_{>0}$ and $b\in \R_{\leq 1}$. Then, for all real numbers  $x\geq b$,  \[x - a\log\max(1,x) =\frac{1}{2}x +\frac{1}{2}(x-2a\log\max(1,x)) \geq \frac{1}{2}x +  \min(\frac{1}{2}b,a-a\log (2a)).\]
\end{lem}
\begin{proof}
It suffices to prove that $x-2a\log\max(1,x) \geq \min(b,2a-2a\log(2a))$ for all $x\geq b$. To prove this, let $x\geq b$. Then, if $2a\leq 1$, we have $x-2a\log \max (1,x)\geq b\geq \min(b,2a-2a\log(2a))$. (To prove that $x-2a \log \max(1,x)\geq b$, we may assume that $x\geq 1$. It is easy to show that $x-2a\log x$ is a non-decreasing function for $x\geq 1$. Therefore, for all $x\geq 1$, we conclude that $x-2a\log x \geq 1 \geq b$.) If $2a>1$, the function  $x-2a\log (x)$ attains its minimum value at $x=2a$ on the interval $[1,\infty)$.
\end{proof}

\begin{lem}{ \bf (Bost)} \label{bost}
Let $X$ be a smooth projective connected curve  over $\Qbar$ of genus $g\geq 1$. Then \[h_{\Fal}(X) \geq -\log(2\pi)g.\]
\end{lem}
\begin{proof}
See \cite[Corollaire 8.4]{GaRe}. (Note that the Faltings height $h(X)$ utilized by Bost, Gaudron and R\'emond is bigger than $h_{\Fal}(X)$ due to a difference in normalization. In fact, we have $h(X) = h_{\Fal}(X) +g\log(\sqrt{\pi})$. In particular, the slightly stronger lower bound $h_{\Fal}(X) \geq -\log(\sqrt{2}\pi)g$ holds.) 
\end{proof}
\begin{lem}\label{logs_plus_h}
Let $X$ be a smooth projective connected curve  over $\Qbar$ of genus $g\geq 1$. Then \[ \log S(X) + h_{\Fal}(X) \geq  \frac{1}{2}h_{\Fal}(X) - (4g^3+5g+1)\log 2 +\min\left(-\frac{g}{2}\log(2\pi), \frac{g}{4}-\frac{g}{4}\log\left(\frac{g}{2}\right)\right).\]
\end{lem}
\begin{proof}
By the explicit formula (\ref{SX})  for $S(X)$   in Section \ref{admissible} and our bounds on theta functions (Lemma \ref{theta}), \[\log S(X) + h_{\Fal}(X) \geq -\frac{g}{4}\log \max(1,h_{\Fal}(X)) -(4g^3+5g+1)\log 2 +h_{\Fal}(X)  .\] Since $h_{\Fal}(X) \geq -g\log(2\pi)$, the statement follows from Lemma \ref{ineq} (with $x=h_{\Fal}(X)$, $a= g/4$ and $b=-g \log(2\pi)$). 
\end{proof}

\begin{lem}\label{Rdejong} Let $X$ be a smooth projective connected curve of genus $g\geq 2$ over $\Qbar$. Then
\[ \frac{(2g-1)(g+1)}{8 (g-1)}e(X) +\frac{1}{8}\delta_{\Fal}(X)  \geq  \log S(X) +h_{\Fal}(X)  .  \]
\end{lem}
\begin{proof}
By  \cite[Proposition 5.6]{deJo},   \begin{eqnarray*} e(X) &\geq & \frac{8(g-1)}{(g+1)(2g-1)}\left(\log R(X)  +h_{\Fal}(X) \right). \end{eqnarray*}  Note that $\log R(X) = \log S(X) - \delta_{\Fal}(X)/8$; see (\ref{Sinvariant}) in Section \ref{admissible}. This implies the inequality.
\end{proof}

\begin{lem} {\bf (Noether formula)}
Let $X$ be a smooth projective connected curve over $\Qbar$ of genus $g\geq 1$. Then \[12 h_{\Fal}(X) = e(X) + \Delta(X)  + \delta_{\Fal}(X) - 4g \log (2\pi).\]
\end{lem}
\begin{proof}
This is well-known; see \cite[Theorem 6]{Faltings1} and \cite[Th\'eor\`eme 2.2]{Moret-Bailly2}.
\end{proof}

\begin{prop}\label{faltingsomega1}
Let $X$ be a smooth projective  connected curve of genus $g\geq 2$ over $\Qbar$. Then
\[\begin{array}{ccc}  h_{\Fal}(X) & \leq  & \frac{(2g-1)(g+1)}{4 (g-1)}e(X)+\frac{1}{4}\delta_{\Fal}(X)+ 20g^3 \\ -g\log(2\pi) & \leq & \frac{(2g-1)(g+1)}{4 (g-1)}e(X)+\frac{1}{4}\delta_{\Fal}(X)+ 20g^3 \\ \Delta(X) & \leq & \frac{3(2g-1)(g+1)}{g-1} e(X) + 2\delta_{\Fal}(X) + 248g^3. \end{array}\]
\end{prop}

\begin{proof}

 Firstly, by Lemma \ref{Rdejong}, 
\[ \frac{(2g-1)(g+1)}{8 (g-1)}e(X) +\frac{1}{8}\delta_{\Fal}(X)  \geq  \log S(X) +h_{\Fal}(X)  .  \] 
To obtain the upper bound for $h_{\Fal}(X)$,  we proceed as follows. By Lemma \ref{logs_plus_h}, \[ \log S(X) + h_{\Fal}(X) \geq  \frac{1}{2}h_{\Fal}(X) - (4g^3+5g+1)\log 2 + \min\left(-\frac{g}{2}\log(2\pi), \frac{g}{4}-\frac{g}{4}\log\left(\frac{g}{2}\right)\right). \] From these two inequalities, we deduce that \[\frac{1}{2}h_{\Fal}(X) \leq \frac{(2g-1)(g+1)}{8 (g-1)}e(X) +\frac{1}{8}\delta_{\Fal}(X)  + (4g^3+5g+1)\log 2 +\max\left(\frac{g}{2}\log(2\pi), \frac{g}{4}\log\left(\frac{g}{2}\right) -\frac{g}{4}\right). \] Finally, it is straightforward to verify the inequality \[(4g^3+5g+1)\log 2 +\max\left(\frac{g}{2}\log(2\pi), \frac{g}{4}\log\left(\frac{g}{2}\right) -\frac{g}{4}\right)\leq 10g^3.\] This concludes the proof of the upper bound for $h_{\Fal}(X)$.

The second inequality follows from the first inequality of the proposition and the lower bound $h_{\Fal}(X)\geq -g\log (2\pi)$ of Bost (Lemma \ref{bost}).

Finally, to obtain the upper bound of the proposition for the discriminant of $X$, we eliminate the Faltings height of $X$ in the first inequality using the Noether formula and obtain \[ \Delta(X) + e(X) + \delta_{\Fal}(X) - 4g\log(2\pi) \leq  \frac{3(2g-1)(g+1)}{ (g-1)}e(X)+3\delta_{\Fal}(X)+ 240g^3. \] In \cite[Theorem 5]{Faltings1} Faltings showed that $e(X)\geq 0$. Therefore, we conclude that \[ \Delta(X)  + \delta_{\Fal}(X) - 4g\log(2\pi) \leq  \frac{3(2g-1)(g+1)}{ (g-1)}e(X)+3\delta_{\Fal}(X)+ 240g^3. \qedhere\] 
\end{proof}

We are now ready to prove Theorem \ref{upperboundinv}. \\

\noindent \emph{Proof of Theorem \ref{upperboundinv}.}
The proof is straightforward.  The upper bound $e(X) \leq 4g(g-1)h(b)$ is well-known; see \cite[Theorem 5]{Faltings1}.

Let us prove the lower bound for $\delta_{\Fal}(X)$.  If $g\geq 2$, the lower bound for $\delta_{\Fal}(X)$ can be deduced from  the second inequality of Proposition \ref{faltingsomega1} and the upper bound  $e(X)\leq 4g(g-1)h(b)$. When $g=1$, this follows from a result of Szpiro (\cite[Proposition 7.2]{deJo2}) and  the non-negativity of $h(b)$.

From now on, we suppose that $b$ is a non-Weierstrass point. The upper bound $h_{\Fal}(X) \leq  \frac{1}{2} g(g+1) h(b) +\log \Vert \mathrm{Wr}\Vert_{\Ar}(b)$  follows from  Theorem 5.9 in \cite{deJo} and (\ref{Wronskian}) in Section \ref{admissible}. 

We deduce the upper bound $\delta_{\Fal}(X)  \leq 6 g(g+1) h(b) +  12\log \Vert\mathrm{Wr}\Vert_{\Ar}(b)+ 4g\log(2\pi)$  as follows. Since $e(X)\geq 0$ and $\Delta(X)\geq 0$, the Noether formula implies that \[\delta_{\Fal}(X) \leq 12 h_{\Fal}(X) + 4g \log(2\pi).\]  Thus, the upper bound for $\delta_{\Fal}(X)$ follows from  the upper bound for $h_{\Fal}(X)$.   

The upper bound $$\Delta(X)\leq 2g (g+1)(4g+1) h(b) + 12  \log \Vert \mathrm{Wr}\Vert_{\Ar} (b) + 93g^3$$ follows from the  inequality $$\Delta(X) \leq 12h_{\Fal}(X) -\delta_{\Fal}(X) + 4g\log(2\pi)$$ and the preceding bounds. (One could also use the last inequality of Proposition \ref{faltingsomega1} to obtain a similar result.) \qed

\section{Bounds for Arakelov-Green functions of Belyi covers}\label{ArakelovGreen}
Our aim is to give explicit bounds for the Arakelov-Green function on a Belyi cover of $X(2)$.
Such bounds have been obtained for certain Belyi covers using spectral methods in \cite{JorKra3}.
The results in \textit{loc. cit.} do not apply to our situation since the smallest positive eigenvalue of the Laplacian can go to zero in a tower of Belyi covers; see \cite[Theorem 4]{Lo}.

Instead, we use a theorem of Merkl to prove explicit bounds for the Arakelov-Green function on a Belyi cover in Theorem \ref{MerklResult}.
More precisely,  we construct a ``Merkl atlas'' for an arbitrary Belyi cover.
Our construction uses an explicit version of a result of Jorgenson and Kramer (\cite{JorKra1}) on the Arakelov $(1,1)$-form due to Bruin.

We use our results to estimate the Arakelov norm of the Wronskian differential in Proposition \ref{Wronskian2}.

Merkl's theorem (\cite[Theorem 10.1]{Merkl}) was used to prove bounds for Arakelov-Green functions of the modular curve $X_1(5p)$ in \cite{EdJo3}.   It is also used by David Holmes \cite{Holmes} to construct ``weak-pseudo-metrics" on hyperelliptic curves.
\subsection{Merkl's theorem}\label{StatingMerkl}
Let $X$ be a compact connected Riemann surface of positive genus and recall that $\mu$ denotes the Arakelov $(1,1)$-form on $X$.

\begin{defn}\label{MerklAtlas} A \emph{Merkl atlas} for $X$ is a quadruple $(\{(U_j,z_j)\}_{j=1}^n, r_1,  M,c_1)$, where $\{(U_j,z_j)\}_{j=1}^n$ is a finite atlas for $X$, $\frac{1}{2}<r_1<1$,  $M\geq 1$ and $c_1>0$ are real numbers such that the following properties are satisfied.
\begin{enumerate}
\item Each $z_j U_j$ is the open unit disc.
\label{hyp:open-unit-disc}
\item The open sets $U_j^{r_1} := \{x\in U_j : \vert z_j(x) \vert < r_1\}$ with $1\leq j\leq n$ cover $X$.
\label{hyp:covering}
\item For all $1\leq j,j^\prime\leq n$, the function $\vert dz_j/dz_{j^\prime}\vert$ on $U_j\cap U_{j^\prime}$ is bounded from above by $M$.
\label{hyp:glueing-function-bound}
\item For $1\leq j \leq n$, write $\mu_{\Ar} = iF_j dz_j \wedge d\overline{z_j}$ on $U_j$. Then $0 \leq F_j(x) \leq c_1$ for all $x\in U_j$.
\label{hyp:mu-bound}
\end{enumerate}
\end{defn}
Given a Merkl atlas  $(\{(U_j,z_j)\}_{j=1}^n, r_1,  M,c_1)$ for $X$, the following result provides explicit bounds for Arakelov-Green functions in $n$, $r_1$, $M$ and $c_1$.

\def\twopiCmerkl{330}

\begin{thm}[Merkl]
\label{Merkl}
Let $(\{(U_j,z_j)\}_{j=1}^n, r_1, M, c_1)$ be a Merkl atlas for $X$. Then
\[ \sup_{X\times X\backslash \Delta} \gr_X \leq \frac{330 n}{(1-r_1)^{3/2}} \log\frac{1}{1-r_1} + 13.2nc_1 +(n-1) \log M. \] Furthermore, for every index $j$ and all $x\neq y\in U_j^{r_1}$, we have  \[ \vert \gr_X(x,y) -\log \vert z_j(x) - z_j(y) \vert \vert \leq \frac{330 n}{(1-r_1)^{3/2}}\log\frac{1}{1-r_1} + 13.2nc_1 +(n-1) \log M. \]
\end{thm}

\begin{proof}
Merkl proved this theorem without explicit constants and without the dependence on $r_1$ in \cite{Merkl}. A proof of the theorem in a more explicit form was given by P. Bruin in his master's thesis. This proof is reproduced, with minor modifications, in the appendix.
\end{proof}

\subsection{An atlas for a Belyi cover of $X(2)$}\label{Atlas}
Let $\h$ denote the complex upper half-plane. Recall that $\SL_2(\R)$ acts on $\h$ via M\"obius transformations. Let $\Gamma(2)$ denote the subgroup of $\textrm{SL}_2(\Z)$ defined as \[\Gamma(2) =\left\{ \left( \begin{smallmatrix} a & b \\ c & d \end{smallmatrix} \right) \in \textrm{SL}_2(\Z) : a \equiv d \equiv 1 \mod 2 \ \textrm{and} \ b\equiv c\equiv 0 \mod  2\right\}.\] The Riemann surface $Y(2) = \Gamma(2)\backslash \h$ is not compact. Let $X(2)$ be the compactification of the Riemann surface $Y(2) = \Gamma(2)\backslash \h$ obtained by adding the cusps $0$, $1$ and $\infty$. Note that $X(2)$ is known as the \textit{compact modular curve associated to the congruence subgroup $\Gamma(2)$ of $\mathrm{SL}_2(\Z)$}. The modular lambda function $\lambda:\h\to \C$ induces an analytic isomorphism $\lambda: X(2)\to \p^1(\C)$; see Section \ref{modular_lambda_function} for details. In particular, the genus of $X(2)$ is zero. For a cusp $\kappa \in \{0,1,\infty\}$, we fix an element $\gamma_\kappa$ in $\mathrm{SL}_2(\Z)$ such that $\gamma_\kappa (\kappa) = \infty$.

We construct an atlas for the compact connected Riemann surface $X(2)$. Let $\dot{B}_\infty$ be the open subset given by the image of the strip  \[\dot{S}_{\infty} := \left\{x+iy : -1\leq x <1, y>\frac{1}{2}\right\} \subset \h\] in $Y(2)$ under the quotient map $\h\longrightarrow \Gamma(2)\backslash \h$ defined by $\tau\mapsto \Gamma(2)\tau$.  The quotient map $\h\longrightarrow \Gamma(2)\backslash \h$ induces a bijection from this strip to $\dot{B}_\infty$. More precisely, suppose that $\tau$ and $\tau^\prime$ in  $\dot{S}_{\infty}$ lie in the same orbit under the action of $\Gamma(2)$. Then, there exists an element \[\gamma = \left( \begin{array}{cc} a & b \\ c & d \end{array} \right) \in \Gamma(2)\] such that $\gamma \tau = \tau^\prime$. If $c\neq 0$, by definition,  $c$ is a non-zero integral multiple of $2$.  Thus,  $c^2 \geq 4$. Therefore, \[ \frac{1}{2} < \Im \tau^\prime =\frac{\Im \tau}{\vert c\tau+d\vert^2} \leq \frac{1}{4 \Im \tau} < \frac{1}{2}.\] This is clearly impossible. Thus,  $c=0$ and $\tau^\prime = \tau \pm b$. By definition, $b=2k$ for some integer $k$. Since $\tau$ and $\tau^\prime$ lie in the above strip,  we conclude  that $b=0$. Thus $\tau=\tau^\prime$. 

Consider the morphism $z_\infty:\h\longrightarrow \C$ given by $\tau \mapsto \exp(\pi i \tau+\frac{\pi}{2})$.  The image of the strip $\dot{S}_{\infty}$  under $z_\infty$ in $\C$ is the punctured open unit disc $\dot{B}(0,1)$.  Now, for any $\tau$ and $\tau^\prime$ in the strip $\dot{S}_{\infty}$, the equality $z_\infty(\tau) =z_\infty(\tau^\prime)$ holds if and only if $\tau^\prime = \tau \pm 2k$ for some integer $k$. But then $k=0$ and $\tau =\tau^\prime$. We conclude that $z_\infty$ factors injectively through $\dot{B}_\infty$. Let $z_\infty:B_\infty\longrightarrow B(0,1)$ denote, by abuse of notation, the induced chart at $\infty$, where $B_\infty := \dot{B}_\infty\cup \{\infty\}$ and $B(0,1)$ is the open unit disc in $\C$. We translate our neighbourhood $B_\infty$ at $\infty$ to a neighborhood for $\kappa$, where $\kappa$ is a cusp of $X(2)$. More precisely, for any $\tau$ in $\h$, define $z_{\kappa}(\tau) = \exp(\pi i \gamma_{k}^{-1}\tau+\pi /2)$.  Let $\dot{B}_\kappa$ be the image of $\dot{S}_{\infty}$ under the map $\h\longrightarrow Y(2)$ given by $\tau\mapsto \Gamma(2)\gamma_\kappa \tau$. We define $B_{\kappa}=\dot{B}_{\kappa}\cup \{\kappa\}$. We let $z_{\kappa}: B_\kappa \to B(0,1)$ denote the induced chart (by abuse of notation).

Since the open subsets $B_{\kappa}$ cover $X(2)$,  we have constructed an atlas $\{(B_\kappa,z_\kappa)\}_{\kappa}$ for $X(2)$, where $\kappa$ runs through the cusps $0$, $1$ and $\infty$.

\begin{defn}\label{belyidef} A \textit{Belyi cover} of $X(2)$ is a morphism of compact connected Riemann surfaces $Y\longrightarrow X(2)$ which is unramified over $Y(2)$. The points of $Y$ not lying over $Y(2)$ are called \emph{cusps}.
\end{defn}

\begin{lem}\label{genusofbelyi}
Let $\pi:Y\longrightarrow X(2)$ be a Belyi cover with $Y$ of genus~$g$. Then, $g\leq \deg \pi$.
\end{lem}
\begin{proof} This is trivial for $g\leq 1$. For $g\geq 2$, the statement follows from the Riemann-Hurwitz formula.
\end{proof}

Let $\pi:Y\longrightarrow X(2)$ be a Belyi cover. We are going to  ``lift'' the atlas $\{(B_\kappa,z_\kappa)\}$ for $X(2)$ to an atlas for $Y$.

Let $\kappa$ be a cusp of $X(2)$.  The branched cover $\pi^{-1}(B_{\kappa}) \longrightarrow B_{\kappa}$ restricts to a finite degree topological cover $\pi^{-1}(\dot{B}_{\kappa}) \longrightarrow \dot{B}_\kappa$. In particular, the composed morphism \[\xymatrix{\pi^{-1}\dot{B}_\kappa \ar[rr] & & \dot{B}_\kappa \ar[rr]^{\sim}_{z_\kappa|_{\dot{B}_\kappa}} & & \dot{B}(0,1) }\] is a finite degree topological cover of $\dot{B}(0,1)$.
 
Recall that the fundamental group of $\dot{B}(0,1)$ is isomorphic to $\Z$. More precisely, for any connected finite degree topological cover $V\to\dot{B}(0,1)$, there is a unique integer $e\geq 1$ such that $V\to \dot{B}(0,1)$ is isomorphic to the cover $\dot{B}(0,1)\longrightarrow \dot{B}(0,1)$ given by $x\mapsto x^e$. 

For every cusp $y$ of $Y$ lying over $\kappa$, let $\dot{V}_y$ be the unique connected component of $\pi^{-1}\dot{B}_\kappa$ whose closure $V_y$ in $\pi^{-1}(B_\kappa)$ contains  $y$. Then, for any cusp $y$, there is a positive integer $e_y$ and an isomorphism $\xymatrix{ w_y:\dot{V}_y \ar[r]^{\sim} & \dot{B}(0,1)}$ such that $w_y^{e_y} = z_\kappa \circ \pi|_{\dot{V}_y}$. The isomorphism $w_y:\dot{V}_y\longrightarrow \dot{B}(0,1)$ extends to an isomorphism $w_y:V_y\longrightarrow B(0,1)$ such that $w_y^{e_y} = z_\kappa \circ \pi|_{V_y}$. This shows that $e_y$ is the ramification index of $y$ over $\kappa$. Note that we have constructed an atlas $\{(V_{y},w_y)\}$ for $Y$, where $y$ runs over the cusps of $Y$.

\subsection{The Arakelov $(1,1)$-form and the hyperbolic metric}\label{cofin}

Let \[\mu_{\mathrm{hyp}}(\tau) = \frac{i}{2} \frac{1}{\Im(\tau)^2}d\tau d\overline{\tau} \] be the hyperbolic metric on $\h$. A Fuchsian group is a discrete subgroup of $\SL_2(\R)$. For any Fuchsian group $\Gamma$, the quotient space $\Gamma\backslash \h$ is a connected Hausdorff topological space and can be made into a Riemann surface in a natural way. The hyperbolic metric $\mu_{\mathrm{hyp}}$ on $\h$ induces a measure on $\Gamma\backslash \h$, given by a smooth positive real-valued $(1,1)$-form outside the set of  fixed points of elliptic elements of $\Gamma$. If the volume of $\Gamma\backslash \h$ with respect to this measure is finite, we call $\Gamma$ a  \emph{cofinite Fuchsian group}.

Let $\Gamma$ be a cofinite Fuchsian group, and let $X$ be the compactification of $\Gamma\backslash \h$ obtained by adding the cusps. We assume that $\Gamma$ has no elliptic elements and that the genus~$g$ of $X$ is positive. There is a unique smooth function $F_\Gamma: X\longrightarrow [0,\infty)$ which  vanishes at the cusps of $\Gamma$ such that \begin{eqnarray}\label{mumuhyp}
 \mu &=& \frac{1}{g} F_\Gamma \mu_{\mathrm{hyp}}.
\end{eqnarray}  A detailed description of $F_\Gamma$ is not necessary for our purposes.
\begin{defn}\label{GammaBelyi}
Let $\pi:Y\longrightarrow X(2)$ be a Belyi cover. Then we define the cofinite Fuchsian group $\Gamma_Y$ (or simply $\Gamma$) associated to $\pi:Y\to X(2)$ as follows.  Since the topological fundamental group of $Y(2)$ equals  $\Gamma(2)/\{\pm 1\} $, we have $\pi^{-1}(Y(2))=\Gamma^\prime\backslash \h$ for some subgroup $\Gamma^\prime\subset \Gamma(2)/\{\pm 1\}$ of finite index. We define $\Gamma \subset \Gamma(2)$ to be the inverse image of $\Gamma^\prime$ under the quotient map $\Gamma(2) \longrightarrow \Gamma(2)/\{\pm 1\}$. Note that $\Gamma$ is a cofinite Fuchsian group without elliptic elements.
\end{defn}

\begin{thm}\label{JK}{\bf (Jorgenson-Kramer)} For any Belyi cover $\pi:Y\longrightarrow X(2)$, where $Y$  has positive genus,   \[  \sup_{\tau \in Y} F_\Gamma  \leq 64 \max_{y\in Y}(e_y)^2 \leq 64 (\deg \pi)^2.\]  \end{thm}
\begin{proof} This is shown in \cite{Bruin0}.
 More precisely, in the notation of \textit{loc. cit.}, Bruin shows that, with $a=1.44$, we have $N_{\mathrm{SL}_2(\Z)}(z,2a^2-1) \leq 58$. In particular, $\sup_{z\in Y} N_\Gamma(z,z,2a^2-1) \leq 58$; see Section 8.2 in \textit{loc. cit.}. Now, we apply Proposition 6.1 and Lemma 6.2 (with $\epsilon = 2\deg \pi$) in \textit{loc. cit.} to deduce the sought inequality.
\end{proof}

\begin{opm}
Jorgenson and Kramer prove a stronger (albeit non-explicit) version of Theorem \ref{JK}; see \cite{JorKra1}.
\end{opm}

\subsection{A Merkl atlas for a Belyi cover of $X(2)$}\label{MerklAtlasSection}
In this section we prove bounds for Arakelov-Green functions of Belyi covers.

Recall that we constructed an atlas $\{(B_\kappa,z_\kappa)\}_{\kappa}$ for $X(2)$. For a cusp $\kappa$ of $X(2)$, let $$y_\kappa:~\HH~\longrightarrow~(0,\infty)$$ be defined by \[\tau \mapsto  \Im(\gamma_{\kappa}^{-1} \tau)=\frac{1}{2}-\frac{\log\vert z_\kappa(\tau)\vert}{\pi}.\] This induces a function $\dot{B}_\kappa\longrightarrow (0,\infty)$ also denoted by $y_\kappa$. 

\begin{lem}\label{transition}
For any two cusps $\kappa$ and $\kappa^\prime$ of $X(2)$, we have  \[\left\vert \frac{dz_\kappa}{dz_{\kappa^\prime}}\right\vert \leq 4\exp(3\pi/2)\] on $B_\kappa \cap B_{\kappa^\prime}$.
\end{lem}
\begin{proof} We work on the complex upper half-plane $\h$. We may and do assume that $\kappa \neq \kappa^\prime$. By applying $\gamma^{-1}_{\kappa^\prime}$, we may and do assume that $\kappa^\prime =\infty$.  On $B_\kappa\cap B_\infty$, we have  \[dz_\kappa(\tau) = \pi i \exp(\pi i \gamma_\kappa^{-1} \tau + \pi/2) d(\gamma_\kappa^{-1} \tau), \quad dz_\infty(\tau) = \pi i \exp(\pi i  \tau + \pi/2) d( \tau).\] Therefore, \[\frac{dz_\kappa}{dz_{\infty}}(\tau) = \exp(\pi i (\gamma_\kappa^{-1}\tau - \tau))\frac{d(\gamma_\kappa^{-1} \tau) }{d( \tau)}.\] It follows from a simple calculation that, for $\gamma_\kappa^{-1} = \left( \begin{array}{cc} a & b \\ c & d\end{array}\right)$ with $c\neq 0$, \[ \left\vert \frac{dz_\kappa}{dz_{\infty}}\right\vert(\tau) =\frac{1}{\vert c \tau+d\vert^2} \exp(\pi(y_\infty(\tau) - y_\kappa(\tau))). \]  For $\tau$ and $\gamma_\kappa^{-1}\tau$ in $B_\infty$, one has $y_\infty(\tau)>1/2$ and $y_\kappa(\tau)>1/2$. From $\vert c\tau +d\vert \geq y_\infty(\tau) = \Im(\tau)$, it follows that \[y_\kappa(\tau)=\Im(\gamma_\kappa^{-1}(\tau)) = \gamma_\infty\left(\frac{a\tau + b }{c\tau + d}\right) = \frac{\Im \tau}{\vert c\tau +d\vert^2} \leq \frac{\Im \tau}{(\Im \tau)^2}\leq 2,  \] and similarly $y_\infty(\tau) \leq 2 $. The statement follows.
\end{proof}

Let $\pi:Y\longrightarrow X(2)$ be a Belyi cover, and let $V=\pi^{-1}(Y(2))$ be the complement of the set of cusps in $Y$. Recall that we constructed an atlas $\{(V_y,w_y)\}$ for $Y$. We assume that the genus~$g$ of $Y$ is positive and, as usual, we let $\mu$ denote the Arakelov $(1,1)$-form on $Y$.

\begin{lem}\label{formula}
For  a cusp $y$ of $\pi:Y\to X(2)$ with $\kappa = \pi(y)$, the equality \[idw_y d\overline{w_y} =  \frac{2\pi^{2}  y_\kappa^2 \vert w_y\vert^{2}}{e_y^2 } \mu_{\mathrm{hyp}}\] holds on $\dot{V}_y$. \end{lem}
\begin{proof} Let $\kappa = \pi(y)$ in $X(2)$.  We work on the complex upper half-plane.
By the chain rule, we have\[d(z_\kappa)=  d(w_y^{e_y}) = e_y w_y^{e_y-1} dw_y.\] Therefore,
\[ e_y^2 \vert w_y\vert^{2e_y-2} dw_y d\overline{w_y} = dz_\kappa  d\overline{z_\kappa}.\]
Note that $dz_\kappa =\pi i z_\kappa d(\gamma_{\kappa}^{-1} ),$ where we view $\gamma_{\kappa}^{-1}:\h\longrightarrow \C$ as a function. Therefore, \[
e_y^2 \vert w_y\vert^{2e_y-2} dw_y d\overline{w_y} = \pi^{2} \vert z_\kappa\vert^{2} d(\gamma^{-1}_\kappa)d(\overline{\gamma^{-1}_\kappa }) .\]
Since $\vert w_y^{e_y}\vert = \vert z_\kappa \vert$, we have   \begin{eqnarray*} i dw_y d\overline{w_y} &=&\frac{i\pi^{2}   \vert w_y\vert^{2}}{e_y^2 }d(\gamma_{\kappa}^{-1})d(\overline{\gamma_\kappa^{-1}}) = \frac{2\pi^{2}  y_\kappa^2 \vert w_y\vert^{2}}{e_y^2 } \frac{i d(\gamma_{\kappa}^{-1})d(\overline{\gamma_\kappa^{-1}})}{2y_\kappa^2}  =\frac{2\pi^{2}  y_\kappa^2 \vert w_y\vert^{2}}{e_y^2 }\left(\mu_{\mathrm{hyp}}\circ \gamma_\kappa^{-1}\right) . \end{eqnarray*} Since $\mu_{\mathrm{hyp}}$ is invariant under the action of $\textrm{SL}_2(\Z)$, this concludes the proof.
 \end{proof}

\begin{prop}\label{fy} Let $y$ be a cusp of $\pi:Y\to X(2)$. Write $\mu  = i F_y dw_y d\overline{w_y}$ on $V_y$. Then $F_y$ is a subharmonic function on $V_y$ and     \[0\leq F_y \leq \frac{128 \exp(3\pi)(\deg \pi)^4}{\pi^2 g}.\]
\end{prop}
\begin{proof} The first statement follows from \cite[page 8]{JorKra1}; see also \cite[page 58]{Bruin}.
The lower bound for $F_y$ is clear from the definition. Let us prove the upper bound for $F_y$.

For a cusp $\kappa$ of $X(2)$, let $\dot{B}_\kappa(2)\subset \dot{B}_\kappa$ be the image of the strip $\{x+iy:-1 \leq x < 1, y>2\}$ in $Y(2)$ under the map $\h\longrightarrow Y(2)$ given by $\tau \mapsto \Gamma(2)\gamma_{\kappa}\tau$. For a cusp $y$ of $Y$ lying over $\kappa$, define $\dot{V}_y(2)=\pi^{-1}(\dot{B}_\kappa(2))$ and $V_y(2) = \dot{V}_y(2)\cup \{y\}$.  Since the boundary $\partial V_y(2)$ of $V_y(2)$ is contained in $V_y - V_y(2)$, by the maximum principle for subharmonic functions, 
\begin{eqnarray*}
 \sup_{ V_y} F_y =  \max( \sup_{ V_y(2)} F_y,  \sup_{V_y-V_y(2)} F_y) = \max(\sup_{\partial V_y(2)} F_y , \sup_{V_y-V_y(2)} F_y) =  \sup_{V_y-V_y(2)} F_y.
\end{eqnarray*}

By Lemma \ref{formula}, Definition \ref{GammaBelyi}  and (\ref{mumuhyp}) in Section \ref{cofin},  \begin{eqnarray}\label{fything} F_y &=& F_\Gamma \frac{e_y^2}{2g\pi^2 y_\kappa^2 \vert w_y\vert^2}.\end{eqnarray} Note that $y_\kappa^{-2} < 4$ on $V_y$.  Furthermore, \[\sup_{V_y-V_y(2)}
\vert w_y\vert^{-2}\leq \sup_{B_\kappa-B_\kappa(2)} \vert z_\kappa\vert^{-2} =  \exp(-\pi) \sup_{B_\kappa-B_\kappa(2)} \exp(2\pi y_\kappa) \leq \exp(3\pi). \] Thus, the proposition  follows from Jorgenson-Kramer's upper bound for $F_\Gamma$ (Theorem \ref{JK}).
\end{proof}

\begin{defn}\label{r1}  Define $s_1 = \sqrt{1/2} $. Note that $\frac{1}{2} <s_1<1$.  For any cusp  $\kappa$ of $X(2)$, let $B_{\kappa}^{s_1}$ be the open subset of $B_{\kappa}$ whose image under $z_\kappa$ is $\{x\in \C : \vert x\vert < s_1\}$.  Moreover,  define the positive real number $r_1$ by the equation $r_1^{\deg \pi} = s_1$. Note that $\frac{1}{2} < r_1 < 1$. For all cusps $y$ of $\pi:Y\to X(2)$, define the subset $V_{y}^{r_1}\subset V_y$ by $V_{y}^{r_1} = \{x\in V_y: \vert w_y(x)\vert < r_1 \}$.
\end{defn}

\begin{thm}\label{MerklResult} Let $\pi:Y\longrightarrow X(2)$ be a Belyi cover such that $Y$ is of  genus $g\geq 1$. Then
\begin{eqnarray*} \sup_{Y\times Y\backslash \Delta} \gr_Y &\leq &  6378027\frac{(\deg \pi)^5}{g}. \end{eqnarray*} Moreover, for every cusp $y$ and all $x\neq x^\prime $ in $V_y^{r_1}$,
\begin{eqnarray*} \left \vert \gr_Y(x,x^\prime) - \log\vert w_{y}(x)-w_y(x^\prime)\vert \right \vert &\leq & 6378027\frac{(\deg \pi)^5}{g}\end{eqnarray*}
\end{thm}
\begin{proof} Write $d=\deg \pi$. Let $s_1$ and $r_1$ be as in Definition \ref{r1}. We define real numbers \[n:=\#(Y-V), \quad M :=4d\exp(3\pi), \quad c_1:= \frac{128 \exp(3\pi)d^4}{\pi^2 g} .\]
Since $n$ is the number of cusps of $Y$, we have $ n \leq 3d$. Moreover \[\frac{1}{1-r_1} \leq \frac{d}{1-s_1}.\] 
Note that  \[ \frac{330 n}{(1-r_1)^{3/2}}\log\frac{1}{1-r_1} + 13.2nc_1 +(n-1) \log M \leq 6378027\frac{d^5}{g}\] Therefore, by Theorem \ref{Merkl}, it suffices to show that \[(\{(V_y,w_y)\}_{y}, r_1,  M, c_1), \] where $y$ runs over the cusps of $\pi:Y\to X(2)$, constitutes a Merkl atlas for $Y$.

 The first condition of Merkl's theorem is satisfied. That is,  $w_y V_{y}$ is  the open unit disc in $\C$.

To verify the second condition of Merkl's theorem, we have to  show that the open sets $V_y^{r_1}$ cover $Y$. For any $x\in V_{y}$, we have   $x\in V_y^{r_1}$ if $\pi(x) \in B_{\kappa}^{s_1}$.   In fact, for any $x$ in $V_{y}$, we have  $\vert w_y( x)\vert < r_1$ if and only if \[\vert z_\kappa(\pi(x))\vert=\vert w_y (x)\vert^{e_y} < r_1^{e_y}.\] Since $r_1<1$, we see that $s_1=r_1^d\leq r_1^{e_y}$. Therefore, if $\pi(x)$ lies in $B_{\kappa}^{s_1}$, we see that $x$ lies in $V_y^{r_1}$. Now, since $s_1 < \frac{\sqrt{3}}{2}$,  we have $X(2) = \cup_{\kappa \in \{0,1,\infty\}} B_\kappa^{s_1}$. Thus, we conclude that $Y=\cup_y V_y^{r_1}$, where $y$ runs through the cusps.

Since we have already verified the fourth condition  of Merkl's theorem in Lemma \ref{fy}, it suffices to verify the third condition to finish the proof. Let $\kappa$ and $\kappa^\prime$ be
cusps of $X(2)$.  We may and do assume that
$\kappa \neq \kappa^\prime$. Now, as usual, we work on the complex upper half-plane.  By the chain rule, 
\[ \left\vert\frac{dw_y}{dw_{y^\prime}}\right\vert \leq  \frac{d}{\vert w_y\vert^{e_y-1}}   \sup_{B_\kappa\cap B_{\kappa^\prime}}\left \vert \frac{dz_\kappa}{dz_{\kappa^\prime}} \right\vert\] on $V_{y}\cap V_{y^\prime}$.
Note that $\vert w_y(\tau)\vert^{e_y-1} \geq \vert w_y(\tau)\vert^{e_y}= \vert z_\kappa(\tau)\vert $ for any $\tau$ in $\h$. Therefore,
\[ \left\vert\frac{dw_y}{dw_{y^\prime}}\right\vert\leq  \frac{d}{\vert z_{\kappa}\vert} \sup_{B_\kappa\cap B_{\kappa^\prime}}\left \vert \frac{dz_\kappa}{dz_{\kappa^\prime}} \right\vert \leq M,\]  where we used Lemma \ref{transition} and the inequality $\vert z_\kappa\vert > \exp(-3\pi /2)$ on $B_\kappa\cap B_{\kappa^\prime}$.
\end{proof}

\subsection{The Arakelov norm of the Wronskian differential}

\begin{prop}\label{Wronskian2} Let $\pi:Y\longrightarrow X(2)$ be a Belyi cover with $Y$ of genus $g\geq 1$. Then \[\sup_{Y-\mathrm{Supp} \mathcal{W}}\log\Vert \mathrm{Wr} \Vert_{\Ar} \leq 6378028 g(\deg \pi)^5.\]
\end{prop}
\begin{proof}
 Let $b$ be a non-Weierstrass point on $ Y$ and let $y$ be a cusp of $Y$ such that $b$ lies in $V_y^{r_1}$. Let  $\omega=(\omega_1,\ldots,\omega_g)$ be an orthonormal basis of $\rH^0(Y,\Omega^1_Y)$.   Then, as  in Section \ref{admissible},  \[\log \Vert \mathrm{Wr}\Vert_{\Ar}(b) = \log\vert W_{w_y}(\omega)(b)\vert+ \frac{g(g+1)}{2}\log\Vert dw_y\Vert_{\Ar}(b).\] By Theorem \ref{MerklResult},  \begin{eqnarray*} \frac{g(g+1)}{2}\log \Vert dw_y\Vert_{\Ar}(b) &\leq & 6378027 g(\deg \pi)^5. \end{eqnarray*} Let us show that $\log \vert W_{w_y}(\omega)(b)\vert \leq  g (\deg \pi)^5$. Write $\omega_k = f_k dw_y$ on $V_y$. Note that $\omega_k \wedge \overline{\omega_k} = \vert f_k\vert^2 dw_y \wedge d\overline{w_y}$. Therefore, \[ \mu = \frac{i}{2g} \sum_{k=1}^g \omega_k \wedge \overline{\omega_k} = \frac{i}{2g} \sum_{k=1}^g \vert f_k\vert^2 dw_y \wedge d\overline{w_y}.\] We deduce that  $\sum_{k=1}^g \vert f_k\vert^2 = 2g F_y$, where $F_y$ is the unique function on $V_y$ such that $\mu = i F_y dw_y \wedge d\overline{w_y}$. By our upper bound for $F_y$ (Proposition \ref{fy}), for any $j=1,\ldots, g$,  \[\sup_{V_y} \vert f_j\vert^2 \leq \sup_{V_y} \sum_{k=1}^g
 \vert f_k\vert^2 =2gF_y \leq  \frac{256\exp(3\pi)(\deg \pi)^4}{\pi^2 }.\]   By Hadamard's inequality, \[ \log\vert W_{w_y}(\omega)(b)\vert \leq \sum_{l=0}^{g-1} \log\left(\sum_{k=1}^g \left\vert \frac{d^{l}f_k}{dw_y^{l}}\right\vert^2(b) \right)^{1/2}.\] Let $r_1 < r<1$ be some real number.  By Cauchy's integral formula, for any $0\leq l \leq g-1$,  \[\left \vert \frac{d^{l}f_k}{dw_y^{l}}\right\vert(b) = \left\vert \frac{l!}{2\pi i} \int_{\vert w_y\vert =r}\frac{f_k}{(w_y-w_y(b))^{l+1}}dw_y\right\vert \leq \frac{l!}{(r-r_1)^{l+1}} \sup_{V_y} \vert f_k\vert \leq  \frac{g!}{(1-r_1)^{g}} \sup_{V_y} \vert f_k\vert. \] By the preceding estimations, since $g! \leq g^g$ and $\frac{1}{1-r_1} \leq \frac{\deg\pi}{1-s_1}$, we obtain that
 \begin{eqnarray*}
\log\vert W_{w_y}(\omega)(b)\vert &\leq & \sum_{l=0}^{g-1} \log\left(\frac{g!}{(1-r_1)^{g}} \left(\sum_{k=1}^g  \sup_{V_y} \vert f_k\vert^2 \right)^{1/2}\right)
\\ & \leq & \sum_{l=0}^{g-1} \log \left( \frac{g!}{(1-r_1)^g}  \left(\sum_{k=1}^g \frac{256 \exp(3\pi)(\deg \pi)^4}{\pi^2 } \right)^{1/2}\right) \\ &=& g\log(g!) +g^2 \log \left(\frac{1}{1-r_1}\right) +\frac{g}{2} \log\left(\frac{ 256g\exp(3\pi)}{\pi^2} \right) + 2g\log(\deg \pi) \\ &\leq & \left(4.5+\log\left(\frac{1}{1-s_1}\right)+\frac{1}{2}\log\left(\frac{256\exp(3\pi)}{\pi^2} \right)\right) g^2 \log (\deg \pi) \\
& \leq & 13 g (\deg \pi)^2.
\end{eqnarray*}   Since $g\geq 1$ and $\pi:Y\to X(2)$ is a Belyi cover, the inequality $\deg \pi \geq 3$ holds. Thus, \[13g (\deg \pi)^2 \leq \frac{13 g (\deg \pi)^5}{27} \leq g (\deg \pi)^5. \qedhere \]
\end{proof}

\section{Points of bounded height}\label{belyiheights}

\subsection{Lenstra's generalization of Dedekind's discriminant bound}\label{lenstra_section}
Let $A$ be a discrete valuation ring of characteristic zero with fraction field $K$.  Let $\mathrm{ord}_A$ denote the valuation on $A$. Let $L/K$ be a finite field extension of degree $n$, and  let $B$ be the integral closure of $A$ in $L$. Note that $L/K$ is separable, and $B/A$ is finite; see \cite[Proposition I.4.8]{Serre}.

The inverse different $\mathfrak{D}_{B/A}^{-1}$ of $B$ over $A$ is the fractional ideal \[\{x\in L : \mathrm{Tr}(xB)\subset A\},\] where $\mathrm{Tr}$ is the trace of $L$ over $K$. The inverse of the inverse different, denoted by $\mathfrak{D}_{B/A}$, is the different of $B$ over $A$. Note that $\mathfrak{D}_{B/A}$ is actually an integral ideal of $L$.

The following proposition (which we would like to attribute to H.W. Lenstra jr.) is a generalization of Dedekind's discriminant bound ( \cite[Proposition III.6.13]{Serre}).

\begin{prop}\label{different0}{ \bf (H.W. Lenstra jr.)} Let $A$ be a discrete valuation ring of characteristic zero with fraction field $K$, and let $B$ be the integral closure of $A$ in a finite field extension $L/K$ of degree $n$. Suppose that $B$ is a discrete valuation ring of ramification index $e$ over $A$. Then, the valuation $r$ of the different ideal $\mathfrak{D}_{B/A}$  on $B$ satisfies the inequality \[ r \leq e - 1+e\cdot \mathrm{ord}_A(n) .\] 
\end{prop}
\begin{proof}   Let $x$ be a uniformiser of $A$. Since $A$ is of characteristic zero, we may define $y:=\frac{1}{nx}$; note that $y$ is an element of $K$. The trace of $y$ (as an element of $L$) is $\frac{1}{x}$. Since $1/x$ is not in $A$, this implies that the inverse different $\mathfrak{D}_{B/A}^{-1}$ is strictly contained in the fractional ideal $yB$. (If not, since $A$ and $B$ are discrete valuation rings, we would have that $yB$ is strictly contained in the inverse different.) In particular, the different $\mathfrak{D}_{B/A}$ strictly contains the fractional ideal $(nx)$. Therefore, the valuation $\mathrm{ord}_B(\mathfrak{D}_{B/A})$ on $B$ of $\mathfrak{D}_{B/A}$  is strictly less than the valuation of $nx$. Thus, \[\mathrm{ord}_B(\mathfrak{D}_{B/A}) < \mathrm{ord}_B(n x)  = e \cdot \mathrm{ord}_A(nx) = e (\mathrm{ord}_A(n) + 1) = e \cdot \mathrm{ord}_A(n) + e .\] This concludes the proof of the inequality. 
\end{proof}

\begin{opm}
If the extension of residue fields of $B/A$ is separable, Proposition  \ref{different0} follows from the \textit{Remarque} following Proposition III.6.13 in \cite{Serre}. (The result in \textit{loc. cit.} was conjectured by Dedekind and proved by Hensel when $A=\Z$.) The reader will see that, in the proof of Proposition \ref{upperbound2}, we have to deal with imperfect residue fields. 
\end{opm}

\begin{prop}\label{different1} Let $A$ be a discrete valuation ring of characteristic zero with fraction field $K$, and let $B$ be the integral closure of $A$ in a finite field extension $L/K$ of degree $n$. Suppose that the residue characteristic $p$ of $A$ is positive.   Let $m$ be the biggest integer such that $p^m\leq n$. Then, for $\beta\subset B$ a maximal ideal of $B$ with ramification index $e_\beta$  over $A$, the valuation  $r_\beta$ of the different ideal $\mathfrak{D}_{B/A}$ at $\beta$ satisfies the inequality \[ r_\beta \leq e_\beta - 1+e_\beta\cdot \mathrm{ord}_A(p^m) .\] 
\end{prop}
\begin{proof} 
To compute $r_\beta$, we localize $B$ at $\beta$, and then take the completions $\widehat{A}$ and $\widehat{B_\beta}$ of $A$ and $B_\beta$, respectively. Let $d$ be the degree of $\widehat{B_\beta}$ over $\widehat{A}$. Then, by Lenstra's result (Proposition \ref{different0}), the inequality \[r_\beta \leq e_\beta -1+e_\beta\cdot \mathrm{ord}_{\widehat{A}}(d). \] holds. By definition,  $\mathrm{ord}_{\widehat A} (d)=\mathrm{ord}_A(d)\leq \mathrm{ord}_A(p^m)$. This concludes the proof. 
\end{proof}

\subsection{Covers of arithmetic surfaces with fixed branch locus}\label{someupperbound}

Let $K$ be a number field with ring of integers $O_K$, and let $S=\Spec O_K$. Let $D$ be a reduced effective divisor on $\mcX = \p^1_{S}$, and let $U$ denote the complement of the support of $D$ in $\mcX$. Let $\mcY \to S$ be an integral normal 2-dimensional flat projective $S$-scheme with geometrically connected fibres,  and let $\pi:\mcY\longrightarrow \mcX$ be a finite surjective morphism of $S$-schemes which is \'etale over $U$. Note that $\pi:\mcY\longrightarrow \mcX$ is a flat morphism. (The source is normal of dimension two, and the target is regular.)  Let  $\psi:\mcY^\prime \to \mcY$ be the minimal resolution of singularities (\cite[Proposition 9.3.32]{Liu2}).  We have the following diagram of morphisms \[\xymatrix{\mcY^\prime \ar[r]^{\psi} & \mcY \ar[r]^\pi & \mcX \ar[r] & S. }\] Consider the prime decomposition $D= \sum_{i\in I} D_i$, where $I$ is a finite index set. Let $D_{ij}$ be an irreducible component of $\pi^{-1}(D)$ mapping onto $D_i$, where $j$ is in the index set $J_i$.  We define $r_{ij}$ to be the valuation of the different ideal of $\mathcal{O}_{\mathcal{Y},D_{ij}}/\mathcal{O}_{\mathcal{X},D_i}$. We define the \textit{ramification divisor} $R$ to be $\sum_{i\in I}\sum_{j\in J_i} r_{ij} D_{ij}$. We define $B:=\pi_\ast R$. (We call $B$ the \textit{branch divisor} of $\pi:\mcY\to \mcX$.)

We apply \cite[6.4.26]{Liu2} to obtain that there exists a dualizing sheaf $\omega_{\mcY/S}$ for $\mcY\to S$, and a dualizing sheaf $\omega_\pi$ for $\pi:\mcY\to \mcX$ such that the adjunction formula \[\omega_{\mcY/S} = \pi^\ast \omega_{\mcX/S}\otimes\omega_{\pi}\] holds. Since the local ring at the generic point of a divisor on $\mcX$ is of characteristic zero, basic properties of the different ideal imply that $\omega_\pi$ is canonically isomorphic to the line bundle $\mathcal{O}_{\mcY}(R)$. We deduce the \textit{Riemann-Hurwitz} formula \[\omega_{\mcY/S} = \pi^\ast \omega_{\mcX/S}\otimes\mathcal{O}_{\mcY}(R).\]

Let $K_{\mcX}=-2\cdot[\infty]$ be the divisor defined by the tautological section of $\omega_{\mcX/O_K}$. Let $K_{\mcY^\prime}$ denote the Cartier divisor on $\mcY^\prime$ defined by the rational section $d(\pi\circ \psi)$ of $\omega_{\mcY^\prime/S}$. We define the Cartier divisor $K_{\mcY}$ on $\mcY$ analogously, i.e., $K_{\mcY}$ is the Cartier divisor on $\mcY$ defined by $d\pi$. Note that $K_{\mcY} = \psi_\ast K_{\mcY^\prime}$. Also, the Riemann-Hurwitz formula implies the following equality of Cartier divisors \[K_\mcY = \pi^\ast K_{\mcX} + R.\]
  
Let $E_1,\ldots, E_s$ be the exceptional components of $\psi: \mathcal{Y}^\prime \longrightarrow \mathcal{Y}$. Note that  the pull-back of the Cartier divisor $\psi^\ast K_{\mathcal{Y}}$  coincides with $K_{\mcY^\prime}$ on $$\mcY^\prime - \bigcup_{i=1}^s E_i.$$ Therefore, there exist integers $c_i$ such that  \[K_{\mathcal{Y}^\prime} = \psi^\ast K_{\mathcal{Y}} + \sum_{i=1}^s c_i E_i,\] where this is an equality of Cartier divisors (\textbf{not only} modulo linear equivalence). Note that $(\psi^\ast K_{\mathcal{Y}},E_i) =0 $ for all $i$. In fact, $K_{\mcY}$ is linearly equivalent to a Cartier divisor with support disjoint from the singular locus of~$\mcY$. 
\begin{lem}\label{c_i}
For all $i=1,\ldots,s$, we have $c_i \leq 0$.
\end{lem}
\begin{proof}
We have the following local statement. Let $y$ be a singular point of $\mcY$, and let $E_1,\ldots, E_r$ be the exceptional components of $\psi$ lying over $y$. We define $$V_+ = \sum_{i=1, c_i >0}^r c_i E_i$$ as the sum on the $c_i>0$. To prove the lemma, it suffices to show that $V_+ =0$. Since the intersection form on the exceptional locus of $\mcY^\prime\to \mcY$ is negative definite (\cite[Proposition~9.1.27]{Liu2}), to prove $V_+ =0$, it suffices to show that $(V_+,V_+) \geq 0$. Clearly, to prove the latter inequality, it suffices to show that, for all $i$ such that $c_i>0$, we have $(V_+,E_i)\geq 0$. To do this, fix $i\in \{1,\ldots,r\}$ with $c_i>0$. Since $\mcY^\prime\to \mcY$ is minimal, we have that $E_i$ is not a $(-1)$-curve. In particular, by the adjunction formula, the inequality $(K_{\mcY^\prime},E_i) \geq 0$ holds. We conclude that \[(V_+, E_i) = (K_{\mcY^\prime},E_i) - \sum_{j=1,c_j<0}^r c_j (E_j,E_i) \geq 0, \] where, in the last inequality, we used that, for all $j$ such that $c_j<0$, we have that  $E_j \neq E_i$.
\end{proof}

\begin{prop}\label{hulp} Let $P^\prime:S\to \mcY^\prime$ be a section, and let  $Q:S\to \mcX$ be the induced section. If the image of $P^\prime$ is not contained in the support of $K_{\mcY^\prime}$, then  \[(  K_{\mcY^\prime},P^\prime)_{\fin} \leq (B,Q)_{\fin}.\]
\end{prop}
\begin{proof}    Note that, by the Riemann-Hurwitz formula, we have $K_{\mathcal{Y}} = \pi^\ast K_{\mathcal{X}} + R$. Therefore, by Lemma \ref{c_i}, we get that
\begin{eqnarray*}
( K_{\mcY^\prime},P^\prime)_{\fin} &=& (\psi^\ast K_{\mcY}+\sum c_i E_i,P^\prime)_{\fin} \\ &=& (\psi^\ast \pi^\ast K_{\mcX}+\psi^\ast R + \sum_{i=1}^s c_i E_i, P^\prime)_{\fin}  \\
& \leq & (\psi^\ast \pi^\ast K_{\mcX},P^\prime)_{\fin}+(\psi^\ast R,P^\prime)_{\fin}.
\end{eqnarray*}
 Since the image of $P^\prime$ is not contained in the support of $K_{\mcY^\prime}$,
we can apply the projection formula for the composed morphism $\pi\circ \psi:\mcY^\prime\to \mcX$
to $( \psi^\ast \pi^\ast K_{\mcX}, P^\prime)_{\fin}$ and $(\psi^\ast R,P^\prime)_{\fin}$; see~\cite[Section~9.2]{Liu2}. This gives \[(K_{\mcY^\prime},P^\prime)_{\fin} \leq (\psi^\ast \pi^\ast K_{\mcX},P^\prime)_{\fin}+(\psi^\ast R,P^\prime)_{\fin}= (K_{\mcX},Q)_{\fin}+ (\pi_\ast R,Q)_{\fin}.\]
Since $K_{\mcX} = -2\cdot [\infty]$, the inequality $(K_{\mcX},Q)_{\fin} \leq 0$ holds. By definition, $B=\pi_\ast R$. This concludes the proof.
\end{proof}

We introduce some notation. For $i$ in $I$ and $j$ in $J_i$, let $e_{ij}$ and $f_{ij}$ be the ramification index and residue degree of $\pi$ at the generic point of $D_{ij}$, respectively.  Moreover, let $\mathfrak{p}_i\subset O_K$ be the maximal ideal corresponding to the image of $D_i$ in $\Spec O_K$. Then, note that $e_{ij}$ is the multiplicity of $D_{ij}$ in the fibre of $\mathcal{Y}$ over $\mathfrak{p}_i$. Now, let $e_{\mathfrak{p}_i}$ and $f_{\mathfrak{p}_i}$ be the ramification index and residue degree of $\mathfrak{p}_i$ over $\Z$, respectively. Finally, let $p_i$ be the residue characteristic of the local ring at the generic point of $D_i$ and, if $p_i>0$, let $m_i$ be the biggest integer such that $p_i^{m_i} \leq \deg \pi$, i.e., $m_i = \lfloor \log (\deg \pi)/\log (p_i)\rfloor$.

\begin{lem}\label{different_thesis} Let $i$ be in $I$ such that $0<p_i \leq \deg \pi$. Then, for all $j$ in $J_i$, \[r_{ij} \leq 2e_{ij}m_i e_{\mathfrak{p}_i}.\]
\end{lem}
\begin{proof} Let $\mathrm{ord}_{D_i}$ be the valuation on the local ring at the generic point of $D_i$. Then, by Proposition \ref{different1}, the inequality \[r_{ij} \leq e_{ij} - 1 +e_{ij} \cdot \mathrm{ord}_{D_i}(p_i^{m_i})\] holds. Note that $\mathrm{ord}_{D_i}(p_i^{m_i}) = m_i e_{\mathfrak{p}_i}$. Since $p_i \leq \deg \pi$, we have that $m_i \geq 1$. Therefore, \[r_{ij}\leq e_{ij} - 1 +e_{ij} m_i e_{\mathfrak{p}_i} \leq 2e_{ij} m_i e_{\mathfrak{p}_i}.\qedhere\]
\end{proof}

Let us introduce a bit more notation. Let $I_1$ be the set of $i$ in $I$ such that $D_i$ is horizontal (i.e., $p_i =0$) or $p_i > \deg \pi $.  Let $D_1 = \sum_{i\in I_1} D_i$. We are now finally ready to combine our results  to bound the ``non-archimedean'' part of the height of a point. 
\begin{prop}\label{upperbound2} Let $P^\prime:S\to \mcY^\prime$ be a section, and let  $Q:S\to \mcX$ be the induced section. If the image of $P^\prime$ is not contained in the support of $K_{\mcY^\prime}$, then  \[(  K_{\mcY^\prime},P^\prime)_{\fin} \leq \deg \pi(D_1,Q)_{\fin} + 2(\deg \pi)^2 \log(\deg \pi)[K:\Q].\]
\end{prop}
\begin{proof}   
Note that \[B= \sum_{i\in I} \left(\sum_{j\in J_i} r_{ij} f_{ij}\right) D_i.\] Let $I_2 $ be the complement of $I_1$ in $I$.  Let  $D_2 = \sum_{i\in I_2} D_i$, and note that $D= D_1+D_2$.  In particular, \begin{eqnarray*}(B,Q)_{\fin} &=& \sum_{i\in I} \sum_{j\in J_i} r_{ij}f_{ij} (D_i,Q)_{\fin} \\ &= & \sum_{i\in I_1} \sum_{j\in J_i} r_{ij}f_{ij} (D_i,Q)_{\fin} + \sum_{i\in I_2}\sum_{j\in J_i} r_{ij}f_{ij} (D_{i},Q)_{\fin}. \end{eqnarray*} Note that, for all $i$ in $I_1$ and $j$ in $J_i$, the ramification of $D_{ij}$ over $D_i$ is tame, i.e.,  the equality $r_{ij} = e_{ij}-1$ holds. 
Note that, for all $i$ in $I$, we have $\sum_{j\in J_i} e_{ij} f_{ij} = \deg \pi$. Thus, 
\[ \sum_{i\in I_1} \sum_{j\in J_i} r_{ij}f_{ij} (D_i,Q)_{\fin}  \leq \sum_{i\in I_1} \sum_{j\in J_i} e_{ij}f_{ij} (D_i,Q)_{\fin} = \deg \pi (D_1,Q)_{\fin}. \] We claim that \[\sum_{i\in I_2}\sum_{j\in J_i} r_{ij}f_{\ij} (D_{i},Q)_{\fin} \leq 2(\deg\pi)^2\log(\deg \pi)[K:\Q]. \] In fact, since, for all $i$ in $I_2$ and $j$ in $J_i$, by Proposition \ref{different_thesis}, the inequality \[r_{ij} \leq 2e_{ij} m_i e_{\mathfrak p_i}\] holds, we have that \begin{eqnarray*} \sum_{i\in I_2}\sum_{j\in J_i} r_{ij}f_{ij} (D_{i},Q)_{\fin} &\leq & 2\sum_{i\in I_2} m_ie_{\mathfrak{p}_i}(D_{i},Q)_{\fin}\left(\sum_{j\in J_i} e_{ij}  f_{ij}\right) \\ &=& 2(\deg \pi)\sum_{i\in I_2} m_i e_{\mathfrak{p}_i} (D_i,Q)_{\fin}.\end{eqnarray*} Note that $(D_i,Q) = \log (\# k(\mathfrak{p}_i)) = f_{\mathfrak p_i}\log p_i$. We conclude that \begin{eqnarray*} \sum_{i\in I_2} m_i e_{\mathfrak{p}_i} (D_i,Q)_{\fin} &=& \sum_{p \textrm{ prime}} \left(\sum_{i\in I_2, p_i =p} e_{\mathfrak{p}_i}f_{\mathfrak{p}_i}\right) \left\lfloor \frac{\log(\deg \pi)}{\log p}\right\rfloor \log(p) \\ &=&  [K:\Q] \sum_{\mathcal X_p\cap \vert D_2\vert \neq \emptyset} \left\lfloor \frac{\log(\deg \pi)}{\log p}\right\rfloor \log(p), \end{eqnarray*} where the last sum runs over all prime numbers $p$ such that the fibre $\mathcal X_p$ contains an irreducible component of the support of $D_2$. Thus, 
\[(B,Q)_{\fin} \leq(\deg \pi)(D_1,Q)_{\fin} + 2(\deg \pi) [K:\Q] \sum_{\mathcal X_p\cap D_2\neq \emptyset} \left\lfloor \frac{\log(\deg \pi)}{\log p}\right\rfloor \log(p). \] Note that \[\sum_{\mathcal X_p\cap D_2\neq \emptyset} \left\lfloor \frac{\log(\deg \pi)}{\log p}\right\rfloor \log(p) \leq \sum_{\mathcal X_p\cap D_2 \neq \emptyset} \log(\deg \pi) \leq \deg\pi \log(\deg \pi),\] where we used that $\mathcal X_p \cap D_2 \neq \emptyset$ implies that $p\leq \deg \pi$.
In particular, \[(B,Q)_{\fin} \leq (\deg \pi) (  D_1,Q)_{\fin} + 2(\deg \pi)^2 \log(\deg \pi) [K:\Q].\] By Proposition \ref{hulp}, we conclude that  \[(K_{\mcY^\prime},P^\prime)_{\fin} \leq (\deg \pi) (  D_1,Q)_{\fin} + 2(\deg \pi)^2 \log(\deg \pi) [K:\Q].\qedhere\]
\end{proof}

\subsection{Models of covers of curves}
In this section, we give a general construction for a model of a cover of the projective line.  Let $K$ be a number field with ring of integers $O_K$, and let $S=\Spec O_K$.

\begin{prop}\label{semi_stable_generalization} Let $\mathcal Y\to \Spec O_K$ be a flat projective morphism with geometrically connected fibres of dimension one, where $\mcY$ is an integral normal scheme. Then, there exists a finite field extension $L/K$ such that the minimal resolution of singularities of the normalization of $\mathcal Y \times_{O_K} O_L$ is  semi-stable over $O_L$.
\end{prop}
\begin{proof} 
This follows from \cite[Corollary 2.8]{Liu1}.
\end{proof}
The main result of this section reads as follows. 

\begin{thm}\label{model} Let $K$ be a number field, and let $Y$ be a smooth projective geometrically connected curve over $K$. Then, for any finite morphism $\pi_K:Y\to \p^1_K$, there exists a number field $L/K$ such that:
\begin{itemize}
\item the normalization $\pi:\mcY\to \p^1_{O_L}$  of $\p^1_{O_L}$ in the function field of $Y_L$  is finite flat surjective;
\item the minimal resolution of singularities $\psi:\mathcal{Y}^\prime \longrightarrow \mathcal{Y}$ is semi-stable over $O_L$;
\item  each irreducible component of the vertical part of the branch locus of the finite flat morphism $\pi:\mcY\to \p^1_{O_L}$ is of characteristic less than or equal to $\deg \pi$. (The characteristic of a prime divisor $D$ on $\p^1_{O_L}$ is the residue characteristic of the local ring at the generic point of $D$.)
\end{itemize}
\end{thm}
\begin{proof}   By Proposition \ref{semi_stable_generalization}, there exists a finite field extension $L/K$ such that the minimal resolution of singularities $\psi:\mathcal{Y}^\prime \longrightarrow \mathcal{Y}$ of the normalization of $\p^1_{O_L}$ in the function field of $Y_L$ is semi-stable over $O_{L}$. Note that the finite morphism $\pi:\mcY\to \p^1_{O_L}$ is flat. (The source is normal of dimension two, and the target is regular.) Moreover, since the fibres of $\mcY^\prime\to \Spec O_L$ are reduced, the fibres of $\mcY$ over $O_L$ are reduced. Let $\mathfrak p\subset O_L$ be a maximal ideal of residue characteristic strictly bigger than $\deg \pi$, and note that the ramification of $\pi:\mcY\to \p^1_{O_L}$ over (each prime divisor of $\p^1_{O_L}$ lying over) $\mathfrak p$ is tame. Since the fibres of $\mcY\to \Spec O_L$ are reduced,  we see that the finite morphism $\pi$ is unramified over $\mathfrak p$. In fact, since $\p^1_{O_L}\to \Spec O_L$ has reduced (even smooth) fibres, the valuation of the different ideal $\mathcal{D}_{\mathcal{O}_D/\mathcal{O}_{\pi(D)}}$ on $\mathcal{O}_D$ of an irreducible component $D$ of $\mathcal Y_{\mathfrak p}$ lying over $\pi(D)$ in $\mcX$ is precisely the multiplicity of $D$ in $\mathcal Y_{\mathfrak p}$. (Here we let $\mathcal{O}_D$ denote the local ring at the generic point of $D$, and $\mathcal{O}_{\pi(D)}$ the local ring at the generic point of $\pi(D)$.) Thus, each irreducible component of the vertical part of the branch locus of $\pi:\mcY\to \p^1_{O_L}$ is of characteristic less or equal to $\deg \pi$.
\end{proof}

\subsection{The modular lambda function}\label{modular_lambda_function}

The modular function $\lambda:\h\to \C$ is defined as
\[\lambda(\tau)  =\frac{\mathfrak{p}\left(\frac{1}{2}+\frac{\tau}{2}\right)- \mathfrak{p}\left(\frac{\tau}{2} \right)}{ \mathfrak{p}\left(\frac{\tau}{2}\right)-\mathfrak{p}\left(\frac{1}{2}
 \right)},\]  where $\mathfrak{p}$
denotes the Weierstrass elliptic function for the lattice $\Z+\tau \Z$ in $\C$.
The function $\lambda$ is $\Gamma(2)$-invariant. More precisely,  $\lambda$
factors through the $\Gamma(2)$-quotient map $\h\rightarrow Y(2)$ and
an analytic isomorphism $Y(2)\overset{\sim}{\longrightarrow} \C-\{0,1\}$.
Thus, the modular function $\lambda$ induces an analytic isomorphism $X(2) \to \p^1(\C)$.
Let us note that $\lambda(i\infty) =0$, $\lambda(1) = \infty$ and $\lambda(0) =1$.

The restriction of $\lambda$ to the imaginary axis $\{iy: y>0\}$  in $\h$
induces a homeomorphism, also denoted by $\lambda$, from $\{iy: y>0\}$ to
the open interval $(0,1)$ in $\R$. In fact, for $\alpha$ in the open interval $(0,1)$,
\[\lambda^{-1}(\alpha) = i \frac{\mathrm{M}(1,\sqrt{\alpha})}{\mathrm{M}(1,\sqrt{1-\alpha})},\]
where $\mathrm{M}$ denotes the arithmetic-geometric-mean.

\begin{lem}\label{Clambda} For  $\tau$ in $\h$, let $q(\tau) = \exp(\pi i \tau)$ and let $\lambda(\tau) = \sum_{n=1}^\infty a_n q^{n}(\tau)$ be the $q$-expansion of $\lambda$ on $\h$.
Then, for any real number $4/5\leq y\leq 1$, \[ -\log \vert \sum_{n=1}^\infty na_n q^{n}(iy)\vert \leq 2. \]  \end{lem}
\begin{proof}
 Note that \[\sum_{n=1}^\infty na_n q^{n} = q \frac{d\lambda }{d q}. \]  It suffices to show that $\vert q d\lambda/dq\vert \geq 3/20$. We will use the  product formula for $\lambda$.
Namely, \[ \lambda(q) = 16q\prod_{n=1}^\infty f_n(q), \quad  f_n(q) := \frac{1+q^{2n}}{1+q^{2n-1}}.\] Write $ f^\prime_n(q) = df_n(q)/ dq$.
Then, \[q\frac{d\lambda}{d q} = \lambda \left(1 +q\sum_{n=1}^\infty \frac{f_n^\prime (q)}{f_n(q)}\right) = \lambda\left(1+q \sum_{n=1}^\infty \frac{d}{dq} (\log f_n(q))\right). \]  
  Note that, for any positive integer $n$ 
  and  $4/5\leq y \leq 1$, \[\left(\frac{d}{dq} \log f_n(q)\right)(iy)\leq 0. \] Moreover, since $\lambda(i) =1/2$ and $\lambda(0) =1$, the inequality $\lambda(iy) \geq 1/2$ holds for all $0\leq y\leq 1$. Also, for $4/5 \leq y \leq 1$,  \[\left(-q \sum_{n=1}^\infty \frac{d}{dq} \log f_n(q)\right)(iy) \leq \frac{7}{10} .\]
In fact,  
\begin{eqnarray*}
\sum_{n=1}^\infty \frac{d}{dq} \left(\log f_n(q)\right) &=& \sum_{n=1}^\infty \frac{2nq^{2n-1}}{1+q^{2n}} - \sum_{n=1}^\infty \frac{(2n-1)q^{2n-2}}{1+q^{2n-1}}
\end{eqnarray*} It is straightforward to verify that, for all $4/5\leq y\leq 1$, the inequality \[ \sum_{n=1}^\infty \frac{2nq^{2n-1}(iy)}{1+q^{2n}(iy)} - \sum_{n=1}^\infty \frac{(2n-1)q^{2n-2}(iy)}{1+q^{2n-1}(iy)} \geq  \frac{100}{109}\sum_{n=1}^\infty 2n q^{2n-1}(iy) - \sum_{n=1}^\infty (2n-1)q^{2n-2}(iy)\] holds. Finally, utilizing classical formulas for geometric series, for all $4/5\leq y\leq 1$, \begin{eqnarray*}
q(iy)\sum_{n=1}^\infty \frac{d}{dq} \left(\log f_n(q)\right)(iy)
&\geq & q(iy)\left( \frac{200 q(iy)}{109(1-q^2(iy))^2} - \frac{1+q^2(iy)}{(1-q^2(iy))^2}\right)\geq \frac{7}{10}.
\end{eqnarray*}
We conclude that \[\left\vert q \frac{d\lambda}{dq} \right\vert \geq \frac{1}{2}\left(1-\frac{7}{10}\right) = \frac{3}{20}.\qedhere\]
\end{proof}

\subsection{ A non-Weierstrass point with bounded height}\label{heightboundlastsection}

The logarithmic height of a non-zero rational number $a=p/q$ is given by \[h_{\textrm{naive}}(a) = \log\max(\vert p\vert,\vert q\vert),\] where $p$ and $q$ are coprime integers and $q> 0$.
\begin{thm}\label{heightofpoint1}
Let $\pi_{\Qbar}:Y\longrightarrow \p^1_{\Qbar}$ be a finite morphism of degree $d$, where $Y/\Qbar$ is a smooth projective  connected curve of positive genus $g\geq 1$. Assume that $\pi_{\Qbar}:Y\to \p^1_{\Qbar}$ is unramified over $\p^1_{\Qbar} - \{0,1,\infty\}$. Then, for any rational number $0< a \leq 2/3$ and any $b \in Y(\Qbar)$ lying over $a$, \[h(b) \leq 3 h_{\textrm{naive}}(a) d^2 +6378031 \frac{d^5}{g}  .\]
\end{thm}
\begin{proof} 
 By Theorem \ref{model}, there exist a number field $K$ and a model $$\pi_K:Y\longrightarrow \p^1_K$$ for $\pi_{\Qbar}:Y\longrightarrow \p^1_{\Qbar}$ with the following three properties: the minimal resolution of singularities $\psi:\mathcal{Y}^\prime \longrightarrow \mathcal{Y}$ of the normalization $\pi:\mathcal{Y}\longrightarrow \p^1_{O_K}$ of $\p^1_{O_K}$ in $\mcY$  is semi-stable over $O_K$, each irreducible component of the vertical part of the branch locus of $\pi:\mcY \to \p^1_{O_K}$ is  of characteristic less or equal to $\deg \pi$ and every point in the fibre of $\pi_K$ over $a$ is $K$-rational. Also, the morphism $\pi:\mcY\to \p^1_{O_K}$ is finite flat surjective.

Let $b\in Y(K)$ lie over $a$.  Let $P^\prime$ be the closure of $b$ in $\mcY^\prime$.   By Lemma \ref{heightbigger}, the height of $b$ is ``minimal'' on the minimal regular model. That is, \[h(b) \leq \frac{(P^{\prime}, \omega_{\mathcal{Y}^\prime/O_K})}{[K:\Q]}.\]  Recall the following notation from Section \ref{someupperbound}. Let $\mcX = \p^1_{O_K}$. Let $K_{\mcX} = -2 \cdot [\infty]$ be the divisor defined by the tautological section. Let~$K_{\mcY^\prime}$ be the divisor on $\mcY^\prime$ defined by $d(\pi_K)$ viewed as a rational section of $\omega_{\mcY^\prime/O_K}$. Since the support of $K_{\mcY^\prime}$ on the generic fibre is contained in $\pi_K^{-1}(\{0,1,\infty\})$, the section $P^\prime$ is not contained in the support of $K_{\mcY^\prime}$. Therefore, we get that \[h(b) [K:\Q] \leq (P^{\prime}, \omega_{\mathcal{Y}^\prime/O_K})= (P^\prime,K_{\mcY^\prime})_{\fin} + \sum_{\sigma:K\longrightarrow \C} (-\log \Vert d\pi_K\Vert_\sigma)(\sigma(b)).\]  

Let $D$ be the branch locus of $\pi: \mathcal{Y} \longrightarrow \mathcal{X}$ endowed with the reduced closed subscheme structure. Write $D= 0+1+\infty+D_{\textrm{ver}}$, where $D_{\textrm{ver}}$ is the vertical part of $D$. Note that, in the notation of Section \ref{someupperbound}, we have that $D_1 = 0+1+\infty$. Thus, if $Q$ denotes the closure of $a$ in $\mcX$,  by Proposition \ref{upperbound2}, we get   \[( P^\prime,K_{\mcY^\prime})_{\fin} \leq (\deg \pi)(0+1+\infty,Q)_{\fin} + 2(\deg \pi)^2 \log(\deg \pi)[K:\Q].\] Write $a=p/q$, where $p$ and $q$ are coprime positive integers with $q>p$. Note that   \begin{eqnarray*} (0+1+\infty,Q)_{\fin} &=& [K:\Q]\log (p q(q-p)) \\ & \leq & 3\log(q)[K:\Q] \\ &=& 3h_{\textrm{naive}}(a)[K:\Q] .\end{eqnarray*}  We conclude that \[\frac{(P^\prime,K_{\mcY^\prime})_{\fin}}{[K:\Q]} \leq 3h_{\textrm{naive}}(a)(\deg \pi)^2 + 2( \deg \pi)^3.\] 

It remains to estimate $\sum_{\sigma:K\longrightarrow \C} (-\log \Vert d\pi_K\Vert_\sigma)(\sigma(b))$. We will use our bounds for Arakelov-Green functions.

Let $\sigma:K\to \C$ be an embedding. The composition \[\xymatrix{ Y_{\sigma} \ar[rr]^{\pi_{\sigma}} & & \p^1(\C) \ar[rr]^{\lambda^{-1}}  & & X(2)  }\] is a Belyi cover (Definition \ref{belyidef}).  By abuse of notation, let $\pi$ denote the composed morphism $Y_{\sigma}\longrightarrow X(2)$. Note that $\lambda^{-1}(2/3) \approx 0.85i$. In particular, $\Im(\lambda^{-1}(a)) \geq \Im(\lambda^{-1}(2/3)) > s_1$. (Recall that  $s_1 = \sqrt{1/2}$.) Therefore,  the element $\lambda^{-1}(a)$ lies in $\dot{B}_{\infty}^{s_1}$.  Since $V_y^{r_1}\supset V_y\cap \pi^{-1}B_{\infty}^{s_1}$, there is a unique cusp $y$ of $Y_\sigma\to X(2)$ lying over $\infty$ such that  $\sigma(b)$ lies in $ V_y^{r_1}$.


Note that $q =z_\infty\exp(-\pi/2)$. Therefore, since $\lambda = \sum_{j=1}^\infty a_j q^{j}$ on $\h$, \[\lambda \circ \pi= \sum_{j=1}^\infty  a_j \exp(-j\pi/2) (z_\infty\circ \pi)^{j}  = \sum_{j=1}^\infty  a_j \exp(-j\pi/2) w_y^{e_yj} \] on $V_y$. Thus, by the chain rule, \[d(\lambda\circ \pi) =  e_y\sum_{j=1}^\infty  ja_{j} \exp(-j\pi/2) w_y^{e_yj-1} d(w_y). \] By the trivial inequality $e_y \geq 1$, the inequality $\vert w_y \vert \leq 1$ and Lemma \ref{Clambda}, 
\begin{eqnarray*}
-\log \Vert d(\lambda\circ \pi)\Vert_{\Ar}(\sigma(b)) &=&  -\log \Vert dw_y\Vert_{\Ar}(\sigma(b)) - \log \vert e_y\sum_{j=1}^\infty ja_j \exp(-j\pi/2) w_y^{e_y j -1}(\sigma(b))\vert \\
& \leq & -\log \Vert dw_y\Vert_{\Ar}(\sigma(b)) - \log \vert \sum_{j=1}^\infty ja_j \exp(-j\pi/2) w_y^{e_y j}(\sigma(b))\vert \\  &\leq & -\log \Vert dw_y\Vert_{\Ar}(\sigma(b)) +2 .  \end{eqnarray*} Thus, by Theorem \ref{MerklResult}, we conclude that \[ \frac{\sum_{\sigma:K\to \C}  (-\log \Vert d\pi_K\Vert_\sigma)(\sigma(b))}{[K:\Q]} \leq 6378027\frac{(\deg \pi)^5}{g} +2. \qedhere\]
  
\end{proof}

\begin{thm}\label{heightboundlast}
Let $Y$ be a smooth projective connected curve over $\Qbar$ of genus $g\geq 1$.
For any finite morphism $\pi:Y\to \p^1_{\Qbar}$ ramified over exactly three points, there exists a non-Weierstrass point $b$ on $Y$ such that
\[h(b) \leq 6378033\frac{(\deg\pi)^5}{g}.\]
\end{thm}
\begin{proof}  Define the sequence $(a_n)_{n=1}^\infty$ of rational numbers by $a_1 = 1/2$ and $a_n = n/(2n-1)$ for $n\geq 2$.
Note that $1/2 \leq a_n \leq 2/3$, and that $h_{\textrm{naive}}(a_n) \leq \log(2n)$.  We may and do assume that $\pi:Y\to \p^1_{\Qbar}$ is unramified over $\p^1_{\Qbar}-\{0,1,\infty\}$.
By Theorem \ref{heightofpoint1}, for all $x\in \pi^{-1}(\{a_n\})$,
\begin{eqnarray}\label{heighty} h(x) \leq 3\log(2n) (\deg \pi)^2 +  6378031 \frac{(\deg \pi)^5}{g} . \end{eqnarray}
Since the number of Weierstrass points on $Y$ is at most $g^3-g$, there exists an integer $1\leq i \leq (\deg \pi)^2$
such that the fibre $\pi^{-1}(a_i)$ contains a non-Weierstrass point, say $b$.  Applying (\ref{heighty}) to $b$, we conclude that \[h(b) \leq 3\log\left(2 (\deg \pi)^2\right) (\deg \pi)^2 +  6378031 \frac{(\deg \pi)^5}{g}  \leq 2 \frac{(\deg \pi)^5}{g} + 6378031 \frac{(\deg \pi)^5}{g}. \qedhere\]
\end{proof}

\subsection{}\label{proofofmaintheorem}
For a smooth projective connected curve $X$ over $\Qbar$, we let $\deg_{B}(X)$ denote the Belyi degree of $X$. \\

\noindent \emph{Proof of Theorem \ref{mainthm}.}
The inequality $\Delta(X)\geq 0$ is trivial, the lower bound  $e(X)\geq 0$ is due to Faltings (\cite[Theorem 5]{Faltings1}) and the lower bound  $h_{\Fal}(X)\geq -g\log(2\pi)$ is due to Bost (Lemma \ref{bost}).

For the remaining bounds, we proceed as follows. By Theorem \ref{heightboundlast}, there exists a non-Weierstrass point $b$ in $X(\Qbar)$ such that \[h(b) \leq 6378033\frac{\deg_B(X)^5}{g}.\] By our bound on the Arakelov norm of the Wronskian differential in Proposition \ref{Wronskian2}, we have \[\log \Vert \mathrm{Wr}\Vert_{\Ar}(b) \leq 6378028 g \deg_B(X)^5.\] To obtain the theorem, we combine these bounds with Theorem \ref{upperboundinv}. \qed

\section{Computing coefficients of modular forms}\label{modularforms}

Let $\Gamma\subset \mathrm{SL}_2(\Z)$ be a congruence subgroup, and let $k$ be a positive integer. A modular form $f$ of weight $k$ for the group $\Gamma$ is determined by $k$ and its $q$-expansion coefficients $a_m(f)$ for $0 \leq m \leq k \cdot [\mathrm{SL}_2(\Z):\{\pm 1\} \Gamma]/12$. In this section we follow \cite{Bruin2} and give an algorithmic application of the main result of this paper. More precisely, the goal of this section is to complete the proof of the following theorem. The proof is given at the end of this section.

\begin{thm} {\bf (Couveignes-Edixhoven-Bruin)} \label{CoEdBr} Assume the Riemann hypothesis  for $\zeta$-functions of number fields. Then there exists a probabilistic  algorithm that, given
\begin{itemize}
\item a positive integer $k$,
\item a number field $K$,
\item  a congruence subgroup $\Gamma \subset \mathrm{SL}_2(\Z)$,
\item  a modular form $f$ of weight $k$ for $\Gamma$ over $K$, and
\item a positive integer $m$ in factored form,
\end{itemize} computes $a_m(f)$, and whose expected running time is bounded by a polynomial in the length of the input.
\end{thm}
\begin{opm} We should make precise how the number field $K$, the congruence subgroup $\Gamma$ and the modular form $f$ should be given to the algorithm, and how the algorithm returns the coefficient $a_m(f)$.  We should also explain what ``probabilistic'' means in this context. For the sake of brevity, we refer the reader to \cite[p. 20]{Bruin2} for the precise definitions. Following the definitions there, the above theorem becomes a precise statement.
\end{opm} 

\begin{opm}
The algorithm in Theorem \ref{CoEdBr} is due to Bruin, Couveignes and Edixhoven.  Assuming the Riemann hypothesis for $\zeta$-functions
of number fields, it was shown that the algorithm runs in polynomial time for \textbf{certain} congruence subgroups; see \cite[Theorem 1.1]{Bruin2}.  Bruin did not have enough information about the semi-stable bad reduction of the modular curve $X_1(n)$ at primes $p$ such that $p^2$ divides $n$ to show that the algorithm runs in polynomial time.
Nevertheless, our bounds on the discriminant of a curve  can be used to show that the algorithm runs in polynomial time for \textbf{all} congruence subgroups.
\end{opm}
\noindent \emph{Proof of Theorem \ref{CoEdBr}.}  We follow Bruin's strategy \cite[Chapter V.1, p. 165]{Bruin}.  In fact, Bruin notes that,
to assure that the algorithm runs in polynomial time for all congruence subgroups, it suffices to show that, for all positive integers $n$, the discriminant $\Delta(X_1(n))$ is polynomial in $n$ (or equivalently the genus of $X_1(n)$). The latter follows from Corollary \ref{modferwol}. In fact, the Belyi degree of $X_1(n)$ is at most the index of $\Gamma_1(n)$ in $\mathrm{SL}_2(\Z)$. Since $$[\mathrm{SL}_2(\Z):\Gamma_1(n)] = n^2\prod_{p\vert n} (1-1/p^2) \leq n^2 ,$$ we conclude that $\Delta(X_1(n)) \leq 5\cdot 10^8 n^{14} $.  \qed

\section{Bounds for heights of covers of curves}\label{conjecture}
Let $X$ be  a smooth projective connected curve over $\Qbar$. We prove that Arakelov invariants of (possibly ramified) covers of $X$ are polynomial in the degree. Let us be more precise.

\begin{thm}\label{mainthm2} Let $X$ be a smooth projective connected curve over $\Qbar$, let $U$ be a non-empty open subscheme of $X$, let $B_f\subset \p^1(\Qbar)$ be a finite set, and let $f:X\to \p^1_{\Qbar}$  be a finite morphism unramified over $\p^1_{\Qbar} - B_f$. Define  $B:=f(X-U)\cup B_f$. Let $N$ be the number of elements in the orbit of $B$ under the action of $\mathrm{Gal}(\Qbar/\Q)$ and let $H_B$ be the height of $B$ as defined in Section \ref{coversofcurves}. Define
\[ c_B := (4NH_B)^{45N^3 2^{N-2}N!}.\]
Then, for any finite morphism $\pi:Y\to X$ \'etale over $U$, where $Y$ is a smooth projective connected curve over $\Qbar$ of genus $g\geq 1$, 
\[ \begin{array}{ccccc} -\log(2\pi)g & \leq &  h_{\Fal}(Y) & \leq  & 13\cdot 10^6 g c_B(\deg f)^5(\deg \pi)^5  \\
0 &\leq &  e(Y)  & \leq  & 3\cdot 10^7(g-1)c_B (\deg f)^5(\deg \pi)^5 \\
 0 &\leq &  \Delta(Y)  & \leq &  5\cdot 10^8  g^2c_B (\deg f)^5(\deg \pi)^5 \\
-10^8 g^2 c_B (\deg f)^5(\deg \pi)^5  & \leq &  \delta_{\Fal}(Y) & \leq &  2\cdot 10^8 g c_B (\deg f)^5 (\deg \pi)^5.
\end{array} \] 
\end{thm}
\begin{proof} We apply Khadjavi's effective version of Belyi's theorem. More precisely, by \cite[Theorem 1.1.c]{Khadjavi}, there exists a finite morphism $R:\p^1_{\Qbar}\to \p^1_{\Qbar}$ \'etale over $\p^1_{\Qbar} - \{0,1,\infty\}$ such that
$R(B) \subset \{0,1,\infty\}$ and  $$\deg R \leq (4NH_B)^{9N^3 2^{N-2}N!}.$$ Note that the composed morphism \[\xymatrix{R\circ f\circ \pi:Y \ar[rr]^\pi  & & X \ar[r]^f & \p^1_{\Qbar} \ar[r]^R & \p^1_{\Qbar}}\] is unramified over $\p^1_{\Qbar} -\{0,1,\infty\}$. We conclude by applying Theorem \ref{mainthm} to the composition  $R\circ f \circ \pi$.
\end{proof}

Note that Theorem \ref{mainthm2} implies Theorem \ref{mainthmintro} (with $X=\p^1_{\Qbar}$, $B_f$ the empty set and $f:X\to \p^1_{\Qbar}$ the identity morphism.)

In the proof of Theorem \ref{mainthm2}, we used Khadjavi's effective version of Belyi's theorem. Khadjavi's bounds are not optimal; see \cite[Lemme 4.1]{Litcanu} and \cite[Theorem 1.1.b]{Khadjavi} for better bounds when $B$ is contained in $\p^1(\Q)$. Actually, the use of Belyi's theorem makes the dependence on the branch locus enormous in Theorem \ref{mainthm2}. It should be possible to avoid the use of Belyi's theorem and improve the dependence on the branch locus in Theorem \ref{mainthm2}. This is not necessary for our present purposes. 

\begin{opm} Let us mention the quantitative Riemann existence theorem due to Bilu and Strambi; see \cite{BiSt}.
Bilu and Strambi give explicit bounds for the naive logarithmic height  of a cover of $\mathbb{P}^1_{\Qbar}$ with fixed branch locus. Although their bound on the naive height is exponential in the degree, the dependence on the height of the branch locus in their result is logarithmic.
\end{opm}

Let us show that Theorem \ref{mainthmintro} implies the following:

\begin{thm} {\bf (\cite[Conjecture 5.1]{EdJoSc})}\label{EdjS}
Let $U\subset \p^1_\Z$ be a non-empty open subscheme. Then there are integers $a$ and $b$
with the following property. For any prime number $\ell$, and for any connected finite \'etale cover $\pi:V\to U_{\Z[1/\ell]}$,  the Faltings height of the normalization of $\p^1_\Q$ in the function field of $V$ is bounded by $(\deg \pi)^a\ell^b$. \end{thm}

\begin{proof} We claim that this conjecture holds with $b=0$ and an integer $a$ depending only on the generic fibre $U_\Q$ of $U$. In fact, let $\pi:Y\to \p^1_\Q$ denote the normalization of $\p^1_\Q$ in the function field of $V$. Note that $\pi$ is \'etale over $U_\Q$. Let $B=\p^1_\Q - U_\Q\subset \p^1(\Qbar)$ and let $N$ be the number of elements in the orbit of $B$ under the action of $\mathrm{Gal}(\Qbar/\Q)$. By Theorem \ref{mainthmintro},  \[h_{\Fal}(Y):=\sum_{X\subset Y_{\Qbar}} h_{\Fal}(X) \leq (\deg \pi)^a, \] where the sum runs over all connected components $X$ of $Y_{\Qbar}:=Y\times_\Q \Qbar$, and  \[ a = 6+\log \left(13\cdot 10^6 N (4NH_B)^{45N^3 2^{N-2}N!}\right).\] Here we used that, $g\leq N\deg\pi $ and \[13\cdot 10^6 g(4NH_B)^{45N^3 2^{N-2}N!} \leq (\deg \pi)^{1+\log\left(13\cdot 10^6 N(4NH_B)^{45N^3 2^{N-2}N!}\right)}.\]  This concludes the proof.  
\end{proof}

Let us briefly mention the context in which these results will hopefully be applied. Let $S$ be a smooth projective geometrically connected surface over $\Q$. As is explained in Section 5 of \cite{EdJoSc}, it seems reasonable to suspect that, there exists an algorithm which, on input of a prime $\ell$, computes the \'etale cohomology groups $\rH^i(S_{\Qbar,\textrm{et}},\F_\ell)$ with their $\mathrm{Gal}(\Qbar/\Q)$-action in time \textbf{polynomial} in $\ell$ for all $i=0, \ldots,4$.

\section*{Appendix: Merkl's method of bounding Green functions}

\ifamslatex
\centerline{by Peter Bruin}
\else
\leftline{by Peter Bruin}
\fi

\medbreak

\noindent
The goal of this appendix is to prove Theorem~\ref{Merkl}.  Let $X$ be
a compact connected Riemann surface, and let $\mu$ be a smooth
non-negative $(1,1)$-form on~$X$ such that $\int_X\mu=1$.  Let $*$
denote the star operator on 1-forms on~$X$, given with respect to a
holomorphic coordinate $z=x+iy$ by
$$
*dx=dy,\quad *dy=-dx,
$$
or equivalently
$$
*dz=-i\,d\bar z,\quad *d\bar z=i\,dz.
$$
The \emph{Green function\/} for~$\mu$ is the unique smooth function
$$
\gr_\mu\colon X\times X\setminus\Delta\to\R,
$$
with a logarithmic singularity along the diagonal $\Delta$, such that for fixed
$w\in X$ we have, in a distributional sense,
$$
{1\over2\pi}d*d\gr_\mu(z,w)=\delta_w(z)-\mu(z)\quad\hbox{and}\quad
\int_{z\in X\setminus\{w\}}\gr_\mu(z,w)\mu(z)=0.
$$

For all $a,b\in X$, we write $g_{a,b}$ for the unique smooth function
on $X\setminus\{a,b\}$ satisfying
\begin{equation}
d*dg_{a,b}=\delta_a-\delta_b\quad\hbox{and}\quad
\int_{X\setminus\{a,b\}} g_{a,b}\mu=0.
\label{eq:def-gab}
\end{equation}
Then for all $a\in X$, we consider the function $g_{a,\mu}$ on
$X\setminus\{a\}$ defined by
\begin{equation}
g_{a,\mu}(x)=\int_{b\in X\setminus\{x\}}g_{a,b}(x)\mu(b).
\label{eq:def-gamu}
\end{equation}
A straightforward computation using Fubini's theorem shows that this
function satisfies
$$
d*dg_{a,\mu}=\delta_a-\mu\quad\hbox{and}\quad
\int_{X\setminus\{a\}} g_{a,\mu}\mu=0.
$$
This implies that $2\pi g_{a,\mu}(b)=\gr_\mu(a,b)$, where $\gr_\mu$ is
the Green function for~$\mu$ defined above.

We begin by restricting our attention to one of the charts of our
atlas, say $(U,z)$.  By assumption, $z$ is an isomorphism from $U$ to
the open unit disc in ${\bf C}$.  Let $r_2$ and~$r_4$ be real numbers
with
$$
r_1<r_2<r_4<1,
$$
and write
$$
r_3=(r_2+r_4)/2.
$$
We choose a smooth function
$$
\tilde\chi\colon{\bf R}_{\ge0}\to[0,1]
$$
such that $\tilde\chi(r)=1$ for $r\le r_2$ and $\tilde\chi(r)=0$ for
$r\ge r_4$.
We also define a smooth function $\chi$ on~$X$ by putting
$$
\chi(x)=\tilde\chi(|z(x)|)\quad\hbox{for }x\in U
$$
and extending by~0 outside~$U$.  Furthermore, we put
$$
\chi^{\rm c}=1-\chi.
$$

For $0<r<1$, we write
$$
U^r=\{x\in U \ : \ |z(x)|<r\}.
$$
For all $a,b\in U^{r_1}$, the function
$$
f_{a,b}={1\over2\pi}\log\left|{(z-z(a))(\overline{z(a)}z-r_4^2)\over
(z-z(b))(\overline{z(b)}z-r_4^2)}\right|
$$
is defined on $U\setminus\{a,b\}$.  Moreover, $f_{a,b}$ is harmonic on
$U\setminus\{a,b\}$, since the logarithm of the modulus of a
holomorphic function is harmonic.  We extend $\chi^{\rm c}f_{a,b}$ to
a smooth function on $U$ by defining it to be zero in $a$~and~$b$.

We consider the open annulus
$$
A=U^{r_4}\setminus\overline{U^{r_2}}.
$$
Let $(\rho,\phi)$ be polar coordinates on $A$ such that
$z=\rho\exp(i\phi)$.  A straightforward calculation shows that in
these coordinates the star operator is given by
$$
*d\rho=\rho\,d\phi,\quad*d\phi=-{d\rho\over\rho}.
$$
We consider the inner product
$$
\langle\alpha,\beta\rangle_A=\int_A\alpha\wedge *\beta.
$$
on the ${\bf R}$-vector space of square-integrable real-valued 1-forms
on~$A$.  Furthermore, we write
$$
\|\alpha\|_A^2=\langle\alpha,\alpha\rangle_A.
$$

\begin{lem}
\label{maxmin}
For every real harmonic function $g$ on $A$ such that $\|dg\|_A$
exists,
$$
\max_{|z|=r_3}g-\min_{|z|=r_3}g\le{2\sqrt{\pi}\over r_4-r_2}\|dg\|_A.
$$
\end{lem}

\begin{proof}
By the formula for the star operator in polar coordinates,
\begin{align*}
dg\wedge*dg&=(\partial_\rho g\,d\rho+\partial_\phi g\,d\phi)\wedge
(\rho\partial_\rho g\,d\phi-\rho^{-1}\partial_\phi g\,d\rho)\\
&=\bigl((\partial_\rho g)^2+(\rho^{-1}\partial_\phi g)^2\bigr)
\rho\,d\rho\,d\phi.
\end{align*}
Using the mean value theorem, we can bound the left-hand side of the
inequality we need to prove by
\begin{align*}
\max_{|z|=r_3}g-\min_{|z|=r_3}g&\le
\pi\max_{|z|=r_3}|\partial_\phi g|\\
&=\pi|\partial_\phi g|(x)
\quad\hbox{for some $x$ with }|z(x)|=r_3.
\end{align*}
We write $R=(r_4-r_2)/2$, and we consider the open disc
$$
D=\bigl\{z\in U\bigm| |z-z(x)|<R\bigr\}
$$
of radius $R$ around $x$; this lies in~$A$ because $r_3=(r_4+r_2)/2$.
Let $(\sigma,\psi)$ be polar coordinates on~$D$ such that
$z-z(x)=\sigma\exp(i\psi)$.  Because $g$ is harmonic, so is
$\partial_\phi g$, and Gauss's mean value theorem implies that
$$
\partial_\phi g(x)={1\over\pi R^2}
\int_D\partial_\phi g\,\sigma\,d\sigma\,d\psi.
$$
On the space of real continuous functions on $D$, we have the inner
product
$$
(h_1,h_2)\mapsto\int_D h_1h_2\,\sigma\,d\sigma\,d\psi.
$$
Applying the Cauchy--Schwarz inequality with
$h_1=\rho^{-1}\partial_\phi g$ and $h_2=\rho$ gives
\begin{align*}
\left|\int_D\partial_\phi g\,\sigma\,d\sigma\,d\psi\right|&\le
\left[\int_D\left(\rho^{-1}\partial_\phi g\right)^2
\sigma\,d\sigma\,d\psi\right]^{1/2}\cdot
\left[\int_D\rho^2\sigma\,d\sigma\,d\psi\right]^{1/2}\\
&\le\left[\int_A(\rho^{-1}\partial_\phi g)^2
\rho\,d\rho\,d\phi\right]^{1/2}\cdot
\left[\int_D\sigma\,d\sigma\,d\psi\right]^{1/2}\\
&\le\left[\int_A dg\wedge *dg\right]^{1/2}[\pi R^2]^{1/2}\\
&=\sqrt{\pi}\,R\Vert dg\Vert_A.
\end{align*}
Combining the above results finishes the proof.
\end{proof}

\begin{lem}
\label{tildegab}
For all $a,b\in U^{r_1}$, there exists a smooth function $\tilde
g_{a,b}$ on~$X$ such that
$$
d*d\tilde g_{a,b}=\begin{cases}
d*d(\chi^{\rm c}f_{a,b})& \text{on }U\\
0&\text{on }X\setminus\overline{U}.\end{cases}
$$
It is unique up to an additive constant and fulfills
$$
\|d\tilde g_{a,b}\|_A\le\|d(\chi^{\rm c}f_{a,b})\|_A.
$$
\end{lem}

\begin{proof}
First we note that the expression on the right-hand side of the
equality defines a smooth 2-form on~$X$, because $d*d(\chi^{\rm
c}f_{a,b})(z)$ vanishes for $|z|>r_4$; this follows from the choice
of~$\chi$ and the fact that $f_{a,b}$ is harmonic for $|z|>r_1$.
Since moreover $\chi^{\rm c}f_{a,b}=0$ on $U^{r_2}$, we see that the
support of this 2-form is contained in the closed annulus $\bar A$.
By Stokes's theorem,
$$
\int_{\bar A} d*d(\chi^{\rm c}f_{a,b})
=\int_{\partial\bar A}*d(\chi^{\rm c}f_{a,b}).
$$
Notice that $f_{a,b}$ is invariant under the substitution $z\mapsto
r_4^2/\bar z$; this implies that $\partial_\rho f_{a,b}(z)=0$ for
$|z|=r_4$.  Furthermore, $\chi^{\rm c}(z)=1$ and $d\chi^{\rm c}(z)=0$
for $|z|=r_4$, so we see that
$$
d(\chi^{\rm c}f_{a,b})(z)=\chi^{\rm c}(z)df_{a,b}(z)
=(\partial_\phi f_{a,b}\,d\phi)(z)\quad\hbox{if }|z|=r_4.
$$
Likewise, since $\chi^{\rm c}=0$ and $d\chi^{\rm c}(z)=0$ for
$|z|=r_2$,
$$
d(\chi^{\rm c}f_{a,b})(z)=\chi^{\rm c}(z)df_{a,b}(z)=0
\quad\hbox{if }|z|=r_2.
$$
This means that for $z$ on the boundary of $\bar A$,
$$
*d(\chi^{\rm c}f_{a,b})(z)=\begin{cases}
-(\partial_\phi f_{a,b}\,d\rho)(z)& \text{if }|z|=r_4\\
0& \text{if }|z|=r_2.\end{cases}
$$
In particular, $*d(\chi^{\rm c}f_{a,b})$ vanishes when restricted to
the submanifold $\partial\bar A$ of~$X$.  From this we conclude that
$$
\int_{\bar A}d*d(\chi^{\rm c}f_{a,b})=
\int_{\partial\bar A}*d(\chi^{\rm c}f_{a,b})=0.
$$
This implies that a function $\tilde g_{a,b}$ with the required
property exists.

\bgroup
\def\g{\tilde g_{a,b}}
\def\h{\chi^{\rm c}f_{a,b}}

To prove the inequality $\Vert d\g\Vert_A\le\Vert d(\h)\Vert_A$, we
note that
\begin{align*}
\|d(\h)\|_A^2&=\|d\g+d(\h-\g)\|_A^2\\
&=\|d\g\|_A^2+2\langle d\g,d(\h-\g)\rangle_A+\|d(\h-\g)\|_A^2.
\end{align*}
The last term is clearly non-negative.  Furthermore, integration by
parts using Stokes's theorem gives
\begin{align*}
\langle d\g,d(\h-\g)\rangle_A&=\int_A d\g\wedge *d(\h-\g)\\
&=\int_{\partial\bar A}\g\,{*d(\h-\g)}-\int_A\g\,d*d(\h-\g).
\end{align*}
The second term vanishes because $d*d\g=d*d(\h)$ on~$A$.  From our
earlier expression for $*d(\chi^{\rm c}f_{a,b})(z)$ on the boundary
of~$A$, we see that
$$
\int_{\partial\bar A}\g\,{*d(\h)}=0.
$$
Finally, because $\partial\bar A$ is also the (negatively oriented)
boundary of $X\setminus A$ and because $d*d\g=0$ on $X\setminus A$,
$$
-\int_{\partial\bar A}\g\,{*d\g}=\int_{X\setminus A}d\g\wedge *d\g
\ge 0.
$$
Thus we have
$$
\langle d\g,d(\h-\g)\rangle_A\ge 0,
$$
which proves the inequality.\egroup
\end{proof}

\begin{lem}
\label{maxmintildegab}
Let $\lambda=\max_{r_2\le r\le r_4}|\tilde\chi'(r)|$.  Then
$$
\max_X\tilde g_{a,b}-\min_X\tilde g_{a,b}\le c_3(r_1,r_2,r_4,\lambda),
$$
where
\begin{align*}
c_3(r_1,r_2,r_4,\lambda)&=4\sqrt{\frac{r_4+r_2}{r_4-r_2}}\left(
\frac{\lambda}{2}\log\frac{(r_1+r_4)^2}{(r_2-r_1)(r_4-r_1)}
+{1\over r_2-r_1}+{r_1\over r_4(r_4-r_1)}\right)\\
&\quad+{2\over\pi}\log\frac{(r_1+r_4)^2}{(r_2-r_1)(r_4-r_1)}.
\end{align*}
\end{lem}

\begin{proof}
First, we note that
\begin{align*}
\max_X\tilde g_{a,b}&=\max\biggl\{\sup_{U^{r_3}}\tilde g_{a,b},
\sup_{X\setminus U^{r_3}}\tilde g_{a,b}\biggr\},\\
\min_X\tilde g_{a,b}&=\min\biggl\{\inf_{U^{r_3}}\tilde g_{a,b},
\inf_{X\setminus U^{r_3}}\tilde g_{a,b}\biggr\}.
\end{align*}
Furthermore,
\begin{align*}
\sup_{U^{r_3}}\tilde g_{a,b}&\le\sup_{U^{r_3}}(\tilde g_{a,b}
-\chi^{\rm c}f_{a,b})+\sup_{U^{r_3}}\chi^{\rm c}f_{a,b}\\
&=\max_{|z|=r_3}(\tilde g_{a,b}-\chi^{\rm c}f_{a,b})
+\max_{r_2\le|z|\le r_3}\chi^{\rm c}f_{a,b}
\end{align*}
because of the maximum principle ($\tilde g_{a,b}-\chi^{\rm c}f_{a,b}$
is harmonic on $U$) and because $\chi^{\rm c}(z)=0$ for $|z|<r_2$.  In
the same way, we find
$$
\inf_{U^{r_3}}\tilde g_{a,b}\ge
\min_{|z|=r_3}(\tilde g_{a,b}-\chi^{\rm c}f_{a,b})
+\min_{r_2\le|z|\le r_3}\chi^{\rm c}f_{a,b}.
$$
We extend $\chi f_{a,b}$ to a smooth function on $X\setminus\{a,b\}$
by putting $(\chi f_{a,b})(x)=0$ for $x\not\in U$.  Then $\tilde
g_{a,b}+\chi f_{a,b}$ is harmonic on $X\setminus\{a,b\}$, and the same
method as above gives us
\begin{align*}
\sup_{X\setminus U^{r_3}}\tilde g_{a,b}&\le
\max_{|z|=r_3}(\tilde g_{a,b}+\chi f_{a,b})
-\min_{r_3\le|z|\le r_4}\chi f_{a,b}\\
&\le\max_{|z|=r_3}(\tilde g_{a,b}-\chi^{\rm c}f_{a,b})
+\max_{|z|=r_3}f_{a,b}-\min_{r_3\le|z|\le r_4}\chi f_{a,b}
\end{align*}
and
$$
\inf_{X\setminus U^{r_3}}\tilde g_{a,b}\ge
\min_{|z|=r_3}(\tilde g_{a,b}-\chi^{\rm c}f_{a,b})
+\min_{|z|=r_3}f_{a,b}-\max_{r_3\le|z|\le r_4}\chi f_{a,b}.
$$
These bounds imply that
\begin{align*}
\max_X\tilde g_{a,b}&\le
\max_{|z|=r_3}(\tilde g_{a,b}-\chi^{\rm c}f_{a,b})
+2\sup_A|f_{a,b}|,\\
\min_X\tilde g_{a,b}&\ge
\min_{|z|=r_3}(\tilde g_{a,b}-\chi^{\rm c}f_{a,b})
-2\sup_A|f_{a,b}|,
\end{align*}
and hence
$$
\max_X\tilde g_{a,b}-\min_X\tilde g_{a,b}\le
\max_{|z|=r_3}(\tilde g_{a,b}-\chi^{\rm c}f_{a,b})
-\min_{|z|=r_3}(\tilde g_{a,b}-\chi^{\rm c}f_{a,b})
+4\sup_A|f_{a,b}|.
$$
By Lemma~\ref{maxmin} and Lemma~\ref{tildegab},
\begin{align*}
\max_{|z|=r_3}(\tilde g_{a,b}-\chi^{\rm c}f_{a,b})
-\min_{|z|=r_3}(\tilde g_{a,b}-\chi^{\rm c}f_{a,b})
&\le{2\sqrt{\pi}\over r_4-r_2}
\|d\tilde g_{a,b}-d(\chi^{\rm c}f_{a,b})\|_A\\
&\le{2\sqrt{\pi}\over r_4-r_2}
(\|d\tilde g_{a,b}\|_A+\|d(\chi^{\rm c}f_{a,b})\|_A)\\
&\le{4\sqrt{\pi}\over r_4-r_2}\|d(\chi^{\rm c}f_{a,b})\|_A.
\end{align*}
We have
\begin{align*}
\|d(\chi^{\rm c}f_{a,b})\|_A&\le\|d(\chi^{\rm c})f_{a,b}\|_A
+\|\chi^{\rm c}df_{a,b}\|_A\\
&\le\|\tilde\chi'(\rho)f_{a,b}\,d\rho\|_A
+\|df_{a,b}\|_A\\
&\le\lambda\|d\rho\|_A\sup_A|f_{a,b}|+\|df_{a,b}\|_A.
\end{align*}
Now
\begin{align*}
\|d\rho\|_A^2&=\int_A d\rho\wedge *d\rho\\
&=\int_A\rho\,d\rho\wedge d\phi\\
&=\pi(r_4^2-r_2^2).
\end{align*}
Furthermore, for all $a,b\in U^{r_1}$ we have
$$
|f_{a,b}(z)|={1\over2\pi}\left|
\log|z-z(a)|+\log|\overline{z(a)}z-r_4^2|
-\log|z-z(b)|-\log|\overline{z(b)}z-r_4^2|\right|.
$$
For all $a\in U^{r_1}$ and all $z\in A$, the triangle inequality gives
$$
r_2-r_1<|z-z(a)|<r_4+r_1\quad\hbox{and}\quad
r_4(r_4-r_1)<|\overline{z(a)}z-r_4^2|<r_4(r_4+r_1).
$$
From this we deduce that for all $a,b\in U^{r_1}$,
$$
\sup_A|f_{a,b}|\le {1\over2\pi}\log\frac{(r_1+r_4)^2}{(r_2-r_1)(r_4-r_1)}.
$$
Finally we bound the quantity $\|df_{a,b}\|_A$.  Because $f_{a,b}$ is
a real function, we have
$$
df_{a,b}=\partial_z f_{a,b}\,dz+\overline{\partial_z f_{a,b}}\,d\bar z.
$$
Therefore,
\begin{align*}
\|df_{a,b}\|_A^2&=\int_A df_{a,b}\wedge *df_{a,b}\\
&=2i\int_A|\partial_z f_{a,b}|^2\,dz\wedge d\bar z\\
&=4\int_0^{2\pi}\!\!\!\int_{r_2}^1
  |\partial_z f_{a,b}|^2\,\rho\,d\rho\,d\phi\\
&\le 4\pi(1-r_2^2)\sup_A|\partial_z f_{a,b}|^2.
\end{align*}
A straightforward computation gives
$$
\partial_z f_{a,b}={1\over4\pi}\left({1\over z-z(a)}
+{\overline{z(a)}\over\overline{z(a)}z-r_4^2}-{1\over z-z(b)}
-{\overline{z(b)}\over\overline{z(b)}z-r_4^2}\right).
$$
Our previous bounds for $|z-z(a)|$ and $|\overline{z(a)}z-1|$ yield
$$
\sup_A|\partial_z f_{a,b}|\le{1\over2\pi}
\left({1\over r_2-r_1}+{r_1\over r_4(r_4-r_1)}\right).
$$
From this we obtain
$$
\|df_{a,b}\|_A\le\sqrt{r_4^2-r_2^2\over\pi}
\left({1\over r_2-r_1}+{r_1\over r_4(r_4-r_1)}\right).
$$
Combining the bounds for $\sup_A|f_{a,b}|$ and $\|df_{a,b}\|_A$ yields
the lemma.
\end{proof}

From now on we impose the normalisation condition
$$
\int_X\tilde g_{a,b}\mu=0
$$
on~$\tilde g_{a,b}$ for all $a,b\in U^{r_1}$; this can be attained by
adding a suitable constant to~$\tilde g_{a,b}$.  Then for all $a,b\in
U^{r_1}$, the function $g_{a,b}$ defined earlier is equal to
\begin{equation}
g_{a,b}=\tilde g_{a,b}+\chi f_{a,b}-\int_X\chi f_{a,b}\mu.
\label{eq:gab-expr}
\end{equation}
Indeed, by the definition of~$\tilde g_{a,b}$, the right-hand side
satisfies \eqref{eq:def-gab}.  Furthermore, for all $a\in U^{r_1}$ we
define a smooth function $l_a$ on~$X\setminus\{a\}$ by
$$
l_a=\begin{cases}{\chi\over2\pi}\log|z-z(a)|& \text{on }U\\
0& \text{on }X\setminus\overline{U};\end{cases}
$$
this is bounded from above by ${1\over2\pi}\log(r_4+r_1)$.

\begin{lem}
\label{gablalb}
For all $a,b\in U^{r_1}$, we have
$$
\max_X|g_{a,b}-l_a+l_b|<c_4(r_1,r_2,r_4,\lambda,c_1),
$$
where
$$
c_4(r_1,r_2,r_4,\lambda,c_1)=c_3(r_1,r_2,r_4,\lambda)
+{1\over2\pi}\log{r_4+r_1\over r_4-r_1}
+\left({8\over3}\log2-{1\over4}\right)\frac{c_1}{r_4^2}.
$$
\end{lem}

\begin{proof}
By \eqref{eq:gab-expr} and the definitions of $f_{a,b}$ and $l_a$, we
get
$$
g_{a,b}-l_a+l_b=\tilde g_{a,b}-\int_X\chi f_{a,b}\mu
+{\chi\over2\pi}\log\left|{\overline{z(a)}z-r_4^2\over
\overline{z(b)}z-r_4^2}\right|,
$$
where the last term is extended to zero outside~$U$.  We bound each of
the terms on the right-hand side.  From $\int_X\tilde g_{a,b}\mu=0$
and the non-negativity of~$\mu$ it follows that
$$
\max_X\tilde g_{a,b}\ge0\ge\min_X\tilde g_{a,b}.
$$
Together with the bound for $\max_X\tilde g_{a,b}-\min_X\tilde
g_{a,b}$ from Lemma~\ref{maxmintildegab}, this implies
$$
\max_X|\tilde g_{a,b}|\le c_3(r_1,r_2,r_4,\lambda,c_1).
$$
Because the support of~$\chi$ is contained in $U^{r_4}$, the
hypothesis~\ref{hyp:mu-bound} of Definition~\ref{MerklAtlas} together
with the definition of~$f_{a,b}$ gives
$$
\int_X\chi f_{a,b}\mu=\int_{U^{r_4}}{\chi\over2\pi}\left(
\log\left|\frac{z-z(a)}{r_4}\right|
+\log\left|\frac{\overline{z(a)}z}{r_4^2}-1\right|
-\log\left|\frac{z-z(b)}{r_4}\right|
-\log\left|\frac{\overline{z(b)}z}{r_4^2}-1\right|\right)\mu.
$$
Writing $w=z/r_4$ and $t=z(a)/r_4$, we have
\begin{align*}
\int_{U^{r_4}}{\chi\over2\pi}\log\left|\frac{z-z(a)}{r_4}\right|\mu
&\le{c_1\over2\pi r_4^2}
\int_{\lower1ex\hbox{$\mkern-8mu{|w|<1\atop|w-t|>1}$}}
\mkern-12mu\log|w-t|\,i\,dw\wedge d\bar w.
\end{align*}
We note that $t$ satisfies $|t|<r_1/r_4$; for simplicity, we relax
this to $|t|\le1$.  Then it is easy to see that the above expression
attains its maximum for $|t|=1$; by rotational symmetry we can take
$t=1$.  We now have to integrate over the crescent-shaped domain
$\bigl\{w\in{\bf C}\bigm| |w|<1\hbox{ and }|w-1|>1\bigr\}$,
which is contained in
$\bigl\{1+r\exp(i\phi)\bigm| 1<r<2,2\pi/3<\phi<4\pi/3\bigr\}$.
We get
\begin{align*}
\int_{U^{r_4}}{\chi\over2\pi}\log\left|\frac{z-z(a)}{r_4}\right|\mu
&<{c_1\over\pi}\int_{2\pi/3}^{4\pi/3}\!\!\int_1^2\log(r)\,r\,dr\,d\phi\\
&=\left({4\over3}\log2-{1\over2}\right)c_1.
\end{align*}
In a similar way, we obtain
\begin{align*}
\int_{U^{r_4}}{\chi\over2\pi}\log\left|\frac{z-z(a)}{r_4}\right|\mu&\ge-\frac{c_1}{2r_4^2},\\
\int_{U^{r_4}}{\chi\over2\pi}\log\left|\frac{\overline{z(a)}z}{r_4^2}-1\right|\mu&<\left({4\over3}\log2-{1\over2}\right)\frac{c_1}{r_4^2},\\
\int_{U^{r_4}}{\chi\over2\pi}\log\left|\frac{\overline{z(a)}z}{r_4^2}-1\right|\mu&\ge-\frac{c_1}{4r_4^2}.
\end{align*}
The same bounds hold for~$b$.  Combining everything, we get
$$
\left|\int_X\chi f_{a,b}\mu\right|\le
\left({8\over3}\log2-{1\over4}\right)\frac{c_1}{r_4^2}.
$$
Finally, we have
\begin{align*}
\max_X{\chi\over2\pi}\log\left|
\frac{\overline{z(a)}z-r_4^2}{\overline{z(b)}z-r_4^2}\right|
&\le{1\over2\pi}\sup_{U^{r_4}}
\log\left|\frac{r_4-\overline{z(a)}z/r_4}{r_4-\overline{z(b)}z/r_4}\right|\\
&\le{1\over2\pi}\log\frac{r_4+r_1}{r_4-r_1},
\end{align*}
which finishes the proof.
\end{proof}

We will now apply Lemma~\ref{gablalb}, which holds for any chart
$(U,z)$ satisfying the hypotheses \ref{hyp:open-unit-disc}
and~\ref{hyp:mu-bound} of Definition~\ref{MerklAtlas}, to our atlas
$\{(U_j,z_j)\mid 1\le j\le n\}$.  Besides including the index $j$ in
the notation for the coordinates, we denote by $l_a^{(j)}$ and
$\chi^{(j)}$ the functions $l_a$ and $\chi$ defined for the coordinate
$(U_j,z_j)$.  We obtain the following generalisation of
Lemma~\ref{gablalb} to the situation where $a$~and~$b$ are arbitrary
points of~$X$.

\begin{lem}
\label{gablalb2}
For all $a,b\in X$ and all $j,k$ such that $a\in U_j^{r_1}$ and $b\in
U_k^{r_1}$,
$$
\sup_X\bigl|g_{a,b}-l^{(j)}_{a\vphantom b}+l^{(k)}_b\bigr|
\le c_5(r_1,r_2,r_4,\lambda,n,c_1,M),
$$
where
$$
c_5(r_1,r_2,r_4,\lambda,c_1,n,M)=
n c_4(r_1,r_2,r_4,\lambda,c_1)
+{n-1\over2\pi}\log\left(M\frac{r_4+r_1}{r_2-r_1}\right).
$$
\end{lem}

\begin{proof}
We first show that for any two coordinate indices $j$~and~$k$ and for
all $a\in U_k^{r_1}\cap U_j^{r_1}$,
\begin{equation}
\sup_X\bigl|l_a^{(k)}-l_a^{(j)}\bigr|\le{1\over2\pi}
\log\left(M\frac{r_4+r_1}{r_2-r_1}\right).
\label{eq:star}
\end{equation}
To prove this, let $y\in X$.  We distinguish three cases to prove that
$l_a^{(k)}(y)-l_a^{(j)}(y)$ is bounded from above by the right-hand
side of~\eqref{eq:star}; the inequality then follows by interchanging
$j$ and~$k$.

\smallskip\noindent{\it Case 1:}\enspace Suppose $y\in U_j$ with
$|z_j(y)-z_j(a)|<(r_2-r_1)/M$.  In this case we have
$$
|z_j(y)|<|z_j(a)|+{r_2-r_1\over M}<r_2,
$$
hence $a,y\in U_j^{r_2}$.  Let $[a,y]^j$ denote the line segment
between $a$ and $y$ in the $z_j$-coordinate, i.e.\ the curve in
$U_j^{r_2}$ whose $z_j$-coordinate is parametrised by
$$
\hat z_j(t)=(1-t)z_j(a)+tz_j(y)\quad(0\le t\le 1).
$$
We claim that this line segment also lies inside $U_k^{r_2}$.  Suppose
this is not the case; then, because the `starting point'
$z_j^{-1}\bigl(\hat z_j(0)\bigr)=a$ does lie in $U_k^{r_2}$,
there exists a smallest $t\in(0,1)$ for which the point
$$
y'=z_j^{-1}\bigl(\hat z_j(t)\bigr)\in U_j^{r_2}
$$
lies on the boundary of $U_k^{r_2}$.  It follows from the
hypothesis~\ref{hyp:glueing-function-bound} of
Definition~\ref{MerklAtlas} that
$$
|z_k(y')-z_k(a)|\le M|z_j(y')-z_j(a)|.
$$
On the other hand,
\begin{align*}
|z_j(y')-z_j(a)|&=t|z_j(y)-z_j(a)|\\
&<(r_2-r_1)/M,
\end{align*}
by assumption, and
$$
|z_k(y')-z_k(a)|>r_2-r_1
$$
by the triangle inequality.  This implies
$$
|z_k(y')-z_k(a)|>M|z_j(y')-z_j(a)|,
$$
a contradiction.  Therefore, the line segment $[a,y]^j$ lies inside
$U_j^{r_2}\cap U_k^{r_2}$.  By
hypothesis~\ref{hyp:glueing-function-bound} of
Definition~\ref{MerklAtlas}, we have
$$
|z_k(y)-z_k(a)|\le M|z_j(y)-z_j(a)|.
$$
Because $\chi^{(j)}(y)=\chi^{(k)}(y)=1$, we find
\begin{align*}
l_a^{(k)}(y)-l_a^{(j)}(y)&={1\over2\pi}\log\left|
{z_k(y)-z_k(a)\over z_j(y)-z_j(a)}\right|\\
&\le{1\over2\pi}\log M,
\end{align*}
which is bounded by the right-hand side of~\eqref{eq:star}.

\smallskip\noindent{\it Case 2:}\enspace Suppose $y\not\in U_j$.
Then $l_a^{(j)}(y)=0$, and thus
$$
l_a^{(k)}(y)-l_a^{(j)}(y)=l_a^{(k)}(y)\le{\log(r_4+r_1)\over2\pi}.
$$

\smallskip\noindent{\it Case 3:}\enspace Suppose $y\in U_j$ and
$|z_j(y)-z_j(a)|\ge(r_2-r_1)/M$.  Then
$$
l_a^{(k)}(y)-l_a^{(j)}(y)\le{\log(r_4+r_1)\over2\pi}
-{\chi^{(j)}(y)\over2\pi}\log{r_2-r_1\over M},
$$
which is also bounded by the right-hand side in~\eqref{eq:star}.

By hypothesis~\ref{hyp:covering} of Definition~\ref{MerklAtlas}, the
open sets $U_j^{r_1}$ cover $X$.  Furthermore, $X$ is connected.  For
arbitrary $a,b\in X$ and indices $j$ and $k$ such that $a\in
U_j^{r_1}$ and $b\in U_k^{r_1}$, we can therefore choose a finite
sequence of indices $j=j_1$, $j_2$, \dots, $j_m=k$ with $m\le n$ and
points $a=a_0$, $a_1$, \dots, $a_m=b$ such that $a_i\in
U_{j_i}^{r_1}\cap U_{j_{i+1}}^{r_1}$ for $1\le i\le m-1$.  Using
$$
g_{a,b}=\sum_{i=1}^m g_{a_{i-1},a_i}
$$
we get
\begin{align*}
\sup_X\bigl|g_{a,b}-l_{a\vphantom b}^{(j)}+l_b^{(k)}\bigr|&=\sup_X\left|
\sum_{i=1}^m\left(g_{a_{i-1},a_i}-l_{a_{i-1}}^{(j_i)}
+l_{a_i}^{(j_i)}\right)+\sum_{i=1}^{m-1}\left(
l_{a_i}^{(j_{i+1})}-l_{a_i}^{(j_i)}\right)\right|\\
&\le\sum_{i=1}^m\sup_X
\left|g_{a_{i-1},a_i}-l_{a_{i-1}}^{(j_i)}+l_{a_i}^{(j_i)}\right|
+\sum_{i=1}^{m-1}\sup_X
\left|l_{a_i}^{(j_{i+1})}-l_{a_i}^{(j_i)}\right|.
\end{align*}
The lemma now follows from Lemma~\ref{gablalb} and the
inequality~\eqref{eq:star}.
\end{proof}

\begin{proof}[Proof of Theorem~\ref{Merkl}]
We choose a continuous partition of unity $\{\phi^j\}_{j=1}^n$
subordinate to the covering $\{U_j^{r_1}\}_{j=1}^n$.  Let $a\in X$ and
let $j$ be an index such that $a\in U_j^{r_1}$.  By the definition of
$g_{a,\mu}$ we have
\begin{align*}
g_{a,\mu}(x)-l_a^{(j)}(x)
&=\int_{b\in X}g_{a,b}(x)\mu(b)-l_a^{(j)}(x)\\
&=\sum_{k=1}^n\int_{b\in U_k^{r_1}}\phi^k(b)
  \bigl(g_{a,b}(x)-l_a^{(j)}(x)\bigr)\mu(b)\\
&=\sum_{k=1}^n\int_{b\in U_k^{r_1}}\phi^k(b)
\bigl(g_{a,b}(x)-l^{(j)}_a(x)+l_b^{(k)}(x)\bigr)\mu(b)
-\sum_{k=1}^n\int_{b\in U_k^{r_1}}\phi^k(b)l_b^{(k)}(x)\mu(b).
\end{align*}
In a similar way as in the proof of Lemma~\ref{gablalb}, one can check
check that for every index~$k$ and all $x\in X$ we have
$$
-\frac{c_1}{2}\le
\int_{b\in U_k^{r_1}} \phi^k(b)l_b^{(k)}(x)\mu(b)
\le\left(\frac{4}{3}\log{2}-\frac{1}{2}\right)c_1,
$$
so that
$$
\sup_{x\in X}\left|
\int_{b\in U_k^{r_1}} \phi^k(b)l_b^{(k)}(x)\mu(b)\right|
\le\frac{c_1}{2}.
$$
Together with Lemma~\ref{gablalb2}, this gives the inequality
\begin{align*}
\sup_X \bigl|g_{a,\mu}-l_a^{(j)}\bigr|&\le
c_5(r_1,r_2,r_4,\lambda,c_1,n,M)
\sum_{j=1}^n\int_{b\in U_j^{r_1}}\phi^j(b)\mu(b)
+\sum_{j=1}^n\frac{c_1}{2}\\
&=c_5(r_1,r_2,r_4,\lambda,c_1,n,M)+\frac{nc_1}{2}.
\end{align*}
We also have
\begin{align*}
\sup_X g_{a,\mu}&\le\sup_X\bigl(g_{a,\mu}-l_a^{(j)}\bigr)
+\sup_X l_a^{(j)}\\
&\le\sup_X\bigl(g_{a,\mu}-l_a^{(j)}\bigr)+{\log(r_4+r_1)\over2\pi}.
\end{align*}
By varying the choice of $r_4$ and~$\tilde\chi$, we can let $r_4$ tend
to~1 and $\lambda$ to $1/(1-r_2)$.
This leads to
\begin{align*}
c_3(r_1,r_2,1,1/(1-r_2))&=4\sqrt{\frac{1+r_2}{1-r_2}}\left(
\frac{1}{2(1-r_2)}\log\frac{(r_1+1)^2}{(r_2-r_1)(1-r_1)}
+{1\over r_2-r_1}+{r_1\over 1-r_1}\right)\\
&\qquad+{2\over\pi}\log\frac{(r_1+1)^2}{(r_2-r_1)(1-r_1)},\\
c_4(r_1,r_2,1,1/(1-r_2),c_1)&=c_3(r_1,r_2,1,1/(1-r_2))
+{1\over2\pi}\log{1+r_1\over 1-r_1}
+\left({8\over3}\log2-{1\over4}\right){c_1},\\
c_5&=n c_4(r_1,r_2,r_4,1/(1-r_2),c_1)
+{n-1\over2\pi}\log\left(M\frac{1+r_1}{r_2-r_1}\right).
\end{align*}
We take
$$
r_2=0.39+0.61r_1.
$$
Then for $r_1>1/2$ one can check numerically that
$$
c_5\le 52.4 \frac{n}{(1-r_1)^{3/2}}\log\frac{1}{1-r_1}
+1.60nc_1+\frac{n-1}{2\pi}\log M.
$$
From this the theorem follows.
\end{proof}



\bibliography{refs}{}
\bibliographystyle{plain}

\end{document}